\documentclass{article}

\usepackage{arxiv}

\usepackage[T1]{fontenc}    
\usepackage{hyperref}       
\usepackage{url}            
\usepackage{booktabs}       
\usepackage{amsfonts}       
\usepackage{nicefrac}       
\usepackage{microtype}      
\usepackage{amsmath}
\usepackage{amssymb}
\usepackage{amsmath, amsthm}
\usepackage{graphicx}
\usepackage[colorinlistoftodos]{todonotes}
\usepackage{datetime}
\theoremstyle{definition}
\newtheorem{definition}{Definition}[section]
\newtheorem{theorem}{Theorem}[section]
\newtheorem{lemma}[theorem]{Lemma}
\newtheorem{proposition}[theorem]{Proposition}
\newtheorem{corollary}[theorem]{Corollary}
\newtheorem{example}[theorem]{Example}
\newtheorem{remark}[theorem]{Remark}

\title{On Stoltenberg's quasi-uniform completion}

\author{
  Athanasios Andrikopoulos\thanks{Associate professor  (https://www.ceid.upatras.gr/webpages/faculty/aandriko/)} \\
  Dept. of Computer Engineering and Informatics\\
  University of Patras\\
  Patras, 26504, Greece \\
  \texttt{aandriko@ceid.upatras.gr} \\
   \And
 Ioannis Gounaridis \\
  Dept. of Computer Engineering and Informatics\\
  University of Patras\\
  Patras, 26504, Greece \\
  \texttt{igounaridis@upatras.gr} \\
}

\begin{document}
\maketitle

\begin{abstract}
In this paper, we give a new completion for quasi-uniform spaces which generalizes the completion theories of Doitchinov \cite{doi1} and Stoltenberg \cite{sto}. The presented completion theory is very well-behaved and extends the completion theory of uniform spaces in a natural way. That is, the definition of Cauchy net and the constructed completion coincide with the classical in the case of uniform spaces. The main contribution this completion theory makes is the notion of the cut of nets which generalize the idea of Doitchinov for the notion of $D$-Cauchy net \cite{and}
\end{abstract}


\section{Introduction}
The problems of completeness and completion in quasi-uniform spaces were mainly considered in \cite{and}, \cite{bru}, \cite{CH}, \cite{csa1}, \cite{doi1}, \cite{FL1}, \cite{sal}, \cite{SP}, \cite{sto}.
A satisfactory extension of the completion theory of uniform
spaces to arbitrary quasi-uniform spaces first naturally leads to Cs\'{a}sz\'{a}r's double completeness developed in the realm of syntopogenous spaces \cite{csa1}.
In this direction, Fletcher and Lindgren have introduced in \cite{FL1} the notion of bicompleteness, and they prove that any quasi-uniform space has a bicompletion, called standard bicompletion. It turns out that this concept coincides (for quasi-uniform spaces) with that of double completeness.
Since the idea underlying the bicompletion is basically symmetric, various authors have tried to construct other, possibly non-symmetric completions as it has naturally arisen from the asymmetric character of quasi-uniform spaces.
By definition, the notion of completeness of a quasi-uniform space as well as the construction of the completion depends on the choice of the definition of Cauchy net or Cauchy filter (often nets and filters lead to equivalent theories). The reason for the difficulty to develop a satisfactory non-symmetric completion theory for the class of all quasi-uniform spaces is due to the struggle to approach the notion of Cauchy net (filter) properly from the case of uniform spaces to the case of quasi-uniform spaces. More precisely, since uniform spaces belong to the class of quasi-uniform spaces, according to Doitchinov \cite{doi1}, a notion of Cauchy net in any quasi-uniform space has to be defined in such a manner that this definition provides the properties that convergent nets are Cauchy, and it agrees with the usual definition for uniform spaces. Moreover, the suggested completion must be a monotone operator with respect to inclusion and give rise to the usual uniform completion in the uniform case.
The problem of defining Cauchy nets or filters in quasi-uniform spaces
has been approached by several authors. The problem of defining Cauchy nets or filters in quasi-uniform spaces has been initially approached by Cs\'{a}sz\'{a}r \cite{csa1}, Sieber and Pervin \cite{SP} and Stoltenberg \cite{sto}.
Cs\'{a}sz\'{a}r \cite{csa1} introduced the notion of Cauchy filter in quasi-uniform
spaces and proved that every syntopogenic space can be embedded in complete space.
As it concerns Csaszar's definition of cauchyness. Isbell \cite{isb} noted that the convergent nets were not necessarily Cauchy.
Stoltenberg in \cite{sto} also gave a definition of Cauchy net in quasi-uniform spaces which generalize the definition of Kelly \cite{kel}
for Cauchy sequences in quasi-metric spaces. 
According to Stoltenberg's definition, one can find a Cauchy sequence (net) that 
is very inconvenient to regard this sequence as a potentially convergent one
by completing the space (see \cite[Example 3]{doi}, \cite{GF}.
Sieber and Pervin \cite{SP} gave a definition of Cauchy filter in a quasi- uniform space $(X,\mathcal{U})$ which admits an equivalent definition for nets: A net $(x_a)_{a\in A}$ is Cauchy in $(X,\mathcal{U})$ whenever given $U\in\mathcal{U}$ there is a point $x_{_U}\in X$ and $a_{_0}\in A$ such that
$(x_{_U},x_a)\in U$ for all $a\geq a_{_0}$, $a\in A$.
The definition of Sieber-Pervin has been used by many authors and is usually accepted as the most appropriate way of generalizing the notion of Cauchy net in uniform spaces.
According to Doitchinov, the definition of Sieber-Pervin has a serious flaw. Namely, the sequences (nets) depends not only on its terms but also on some other points which need not belong to it  \cite[Example 2]{doi}.
There are various generalizations to the notion of Cauchy sequence (net) which are based on the definitions of Cs\'{a}sz\'{a}r, Stoltenberg and Sieber and Pervin, but, up to now, none of these generalizations can give a satisfying completion theory for all quasi-uniform spaces. Thus, there exist many different notions of quasi-uniform completeness in the literature. 
More precisely, this problem has been studied in \cite{RSV}, where 7 different notions of \textquotedblleft Cauchy sequences (nets)\textquotedblright are presented.
By combining the 7 \textquotedblleft Cauchy sequences (nets)\textquotedblright with the topologies $\tau_{_{\mathcal{U}}}$, and $\tau_{_{{\mathcal{U}^{-1}}}}$, we may reach a total of 14 different definitions of \textquotedblleft complete space\textquotedblright (considering the symmetry of using the $\mathcal{U}^{-1}$ instead of $\mathcal{U}$). Doitchinov \cite{doi1} developed an interesting completion theory for quiet quasi-metric spaces.
What is interesting in Doitchinov's work is that he has considered a quasi-uniform space as a bitopological space and introduced the concept of conet of a net. 
K\"{u}nzi and Kivuvu have extended the completion theory of Doitchinov of quasi-pseudometric (quasi-uniform) spaces to arbitrary spaces, denoted
$B$-completion.
Andrikopoulos \cite{and1} has introduced a new technique, inspired by Dedekind-MacNeille completion of rational numbers. 

In this paper, we give a completion theory based on Stoltenberg's one which generalizes Doitchinov's completion theory and it satisfies all requirements posed
by him for a good completion.
The technique stands on the construction of a cut of nets, using Doitchinov's
concept of Cauchy pair of nets. In fact, by Proposition \ref{kjz}, to each $D$-Cauchy net it corresponds a set of pairs of nets-conets which lead to the notion of cut of nets, that is, a pair $(\mathcal{C},\mathcal{D})$
where $\mathcal{C}$ contains all equivalent nets of the given net and $\mathcal{D}$ contains all the conets of the members of $\mathcal{C}$.
The notion of $\mathcal{U}$-cut defined in this paper is a cut of nets $(\mathcal{A},\mathcal{B})$ where the members of $\mathcal{A}$ contains right $\mathcal{U}_{_S}$-Cauchy nets and $\mathcal{B}$ contains left $\mathcal{U}_{_S}$-Cauchy nets as they are defined by Stoltenberg.
We call the space $\mathcal{U}$-complete if each member of the first class of a $\mathcal{U}$-cut converges and we prove that each quasi-uniform space has a $\mathcal{U}$-completion.
The new completion eliminates the weaknesses of Stoltenberg's completion.

\section{Notations and definitions}
Let us recall some main notions, basic concepts, definitions and results needed in the paper (see  \cite{mun}, \cite{wil}).
A quasi-pseudometric space $(X,d)$ is a set $X$ together with a non-negative real-valued function $d: X\times X\longrightarrow \mathbb{R}$ (called a quasi-pseudometric) such that, for every $x, y, z\in X$: (i) $d(x,x)=0$;
(ii) $d(x,y)\leq d(x,z)+d(z,y)$. If $d$ satisfies the additional condition (iii) $d(x,y)=0$ implies $x=y$, then $d$ is called a quasi-metric on $X$. A quasi-pseudometric is a pseudometric provided $d(x,y)=d(y,x)$. The conjugate of a
quasi-pseudometric $d$ on $X$ is the quasi-pseudometric $d^{-1}$ given by $d^{-1}(x,y)=d(y,x)$. By $d^{\ast}$ we denote the pseudometric given by $d^{\ast}={\rm max}\{d(x,y),d^{-1}(x,y)\}$. 
Each quasi-pseudometric $d$ on $X$ induces a topology $\tau_{d}$ on $X$ which has as a base the family of $d$-balls $\{ B_{_d}(x,r): x\in X, r>0\}$ where $B_{_d}(x,r)=\{y\in X: d(x,y)<r\}$. A quasi-pseudometric space is $T_{_0}$ if its associated topology $\tau_{d}$ is $T_{_0}$. In that case axiom (i) and the $T_{_0}$-condition can be replaced by (i$^{\prime}$) $\forall x,y\in X$, $d(x,y)=d(y,x)=0 \Leftrightarrow x=y$.
In this case we say that $d$ is a $T_{_0}$-{\it quasi-pseudometric}.
If $\mathcal{P}$ is a family of quasi-pseudometrics on the set $X$, we say that $\mathcal{P}$ is a {\it quasi-gauge}. The topology $\tau(\mathcal{P})$ which has as a subbase the family 
of all balls $B(x,p,\epsilon)$ with $p\in \mathcal{P}$, $x\in X$ and $\epsilon>0$ is called the {\it topology induced on} $X$ by the quasi-gauge $\mathcal{P}$ (see \cite{rei4}).

A quasi-uniformity on a non-empty set $X$ is a filter $\mathcal{U}$ on $X\times X$
which satisfies: (i) $\Delta(X)=\{(x,x)\vert x\in X\}\subseteq U$
for each $U\in \mathcal{U}$ and (ii) given $U\in \mathcal{U}$
there exists $V\in \mathcal{U}$ such that $V\circ V\subseteq U$. The elements of the filter $\mathcal{U}$ are called {\it entourages}.
If $\mathcal{U}$ is a quasi-uniformity on a set $X$, then ${\mathcal{U}}^{-1}=\{U^{-1}\vert U\in \mathcal{U}\}$ is also a
quasi-uniformity on $X$ called the {\it conjugate} of $\mathcal{U}$.
A {\it uniformity} for $X$ is a quasi-uniformity which also satisfies the additional axiom: 
(iii) For all  $U\in \mathcal{U}$ we have $U^{-1}\in\mathcal{U}$ ($\mathcal{U}=\mathcal{U}^{-1}$). The pair $(X,\mathcal{U})$ is called a ({\it quasi}-){\it uniform space}. 
Given a quasi-uniformity $\mathcal{U}$ on $X$, $\mathcal{U}^{\star}=\mathcal{U}\bigvee \mathcal{U}^{-1}$ will denote the coarsest uniformity on $X$ which is finer than $\mathcal{U}$. If $U\in \mathcal{U}$, the entourage $U\cap U^{-1}$ of $\mathcal{U}^{\star}$ will be denoted by $U^{\star}$.
A family $\mathcal{B}$ is a base for a quasi-uniformity $\mathcal{U}$ if and only if for each $U\in \mathcal{U}$ there exists $B\in \mathcal{B}$ such that $B\subseteq U$.
The family $\mathcal{B}$ is  {\it subbase} for $\mathcal{U}$ if the family of finite intersections of members of $\mathcal{B}$ form a base for $\mathcal{U}$.
Every quasi-uniformity $\mathcal{U}$ on $X$ generates a topology $\tau(\mathcal{U})$.
A neighborhood base for each point $x\in X$ is given by $\{U(x)\vert U\in \mathcal{U}\}$ where $U(x)=\{y\in X\vert (x,y)\in U\}$.

If $(X,d)$ is a quasi-pseudometric space then $\mathcal{B}=\{U_{_{d,\epsilon}}\vert \epsilon>0\}$,
where $U_{_{d,\epsilon}}=\{(x,y)\in X\times X\vert d(x,y)<\epsilon\}$, is a base 
for a quasi-uniformity $\mathcal{U}_{_d}$ for $X$ such that $\tau_{_d}=\tau(\mathcal{U}_{_d})$.
For each quasi-uniformity $\mathcal{U}$
possessing a countable base there is a quasi-pseudometric $d_{_{\mathcal{U}}}$ such that $\tau(\mathcal{U})=\tau\!_{_{d_{_{_{\mathcal{U}}}}}}$.

A function $f$ from a quasi-uniform space $(X,\mathcal{U})$ to a quasi-uniform
space $(X,\mathcal{V})$ is {\it quasi-uniformly continuous }
if for each $V\in \mathcal{V}$ there is $U\in \mathcal{U}$ such that $(f(x),f(y))\in \mathcal{V}$ whenever $(x,y)\in U$ i.e. the set $\{(x, y)\vert (f(x),f(y))\in \mathcal{V}\}\in \mathcal{U}$. The function $f$ is a quasi-uniform isomorphism if and only if $f$ is one-to-one onto $Y$ and $f^{-1}$ is quasi-uniformly continuous. A quasi-uniform space $(X,\mathcal{U})$ can
be embedded in a quasi-uniform space $(X,\mathcal{V})$ when there exists a
quasi-uniform isomorphism from $(X,\mathcal{U})$ onto a subspace of $(Y,\mathcal{V})$.
A sequence $(x_{_n})_{_{n\in\mathbb{N}}}$ in a quasi-pseudometric space $(X,d)$ is called $d_{_K}$-{\it Cauchy} (from Kelly \cite[Definition 2.10]{kel}) if for each $\epsilon>0$ there is $k\in\mathbb{N}$ such that $d(x_{_n},x_{_m})<\epsilon$ for each $n\geq m\geq k$. 
This is the notion of Cauchy sequence called in \cite[Definition 1.iv]{RSV} a {\it right} K-{\it Cauchy sequence}.
Similarly, a sequence $(x_{_n})_{_{n\in\mathbb{N}}}$ 
is {\it left} K-{\it Cauchy} if for each $\epsilon>0$ there is $k\in\mathbb{N}$ such that $d(x_{_m},x_{_n})<\epsilon$ for each $n\geq m\geq k$ (see \cite[Definition 1.v]{RSV}). The space $(X,d)$ is said to be {\it right} (resp. {\it left}) K-{\it sequentially complete} if each {\it right} (resp. {\it left}) K-{\it Cauchy} sequence converges in $X$. According to (\cite[Definition 3]{and1})
a sequence $(x_{_n})_{_{n\in \mathbb{N}}}$ on $X$ is {\it right} (resp. {\it left}) $d$-{\it cofinal to} a sequence $(x_{_m})_{_{m\in \mathbb{N}}}$ 
on $X$, if for each $\varepsilon >0$ there exists $n_{_\varepsilon}\in \mathbb{N}$ satisfying the following property: for each $n\geq n_{_\varepsilon}$ there exists $m_{_n}\in \mathbb{N}$ such that $d(x_{_m},x_{_{n}})<\varepsilon$ (resp. $d(x_{_{n}},x_{_m})<\varepsilon$) whenever $m\geq m_{_n}$. The sequences $(x_{_n})_{_{n\in \mathbb{N}}}$ and $(x_{_m})_{_{m\in \mathbb{N}}}$
are {\it right} (resp. {\it left}) $d$-{\it cofinal} if $(x_{_n})_{_{n\in \mathbb{N}}}$ is right (resp. left) $d$-cofinal to $(x_{_m})_{_{m\in \mathbb{N}}}$ and vice versa. According to(\cite[Definition 1]{doi1}) a sequence $(y_{_m})_{_{m\in \mathbb{N}}}$ is called a {\it cosequence} to
$(x_n)_{n\in \mathbb{N}}$, if for any $\varepsilon>0$ there are $n_{_\varepsilon}, m_{_\varepsilon} \in \mathbb{N}$ such that $d(y_{_m},x_n)<\varepsilon$
when $n\geq n_{_\varepsilon}$, $m\geq m_{_\varepsilon}$.
In this case, we write $d(y_{_m},x_n)\rightarrow 0$ or $\lim\limits_{m,n}d(y_{_m},x_n)=0$.
Generally speaking, when two sequences $(x_{_n})_{_{n\in \mathbb{N}}}$ and $(y_{_m})_{_{m\in \mathbb{N}}}$ are given in a quasi-pseudometric space $(X,d)$ we will write $\displaystyle\lim_{n,m}d(y_{_m},x_{_n})=r$ if for any $\varepsilon >0$ there is an $N_{_\varepsilon}$ such that $|d(y_{_m},x_{_n})-r|<\varepsilon$
when $m, n >N_{_\epsilon}$.
We call $\kappa$-{\it cut} {\rm (of sequences) in $X$ (\cite[Definition 8]{and1})} an ordered pair $\xi=(\mathcal{A},{\mathcal{B}})$ of families of right $K$-Cauchy sequences and left $K$-Cauchy cosequences, respectively, with the following propositionerties: (i)
{\rm For any $(x_{_n})_{_{n\in \mathbb{N}}}\in \mathcal{A}$ and any
$(x_{_m}){_{m\in N}}\in {\mathcal{B}}$ there holds $\lim\limits_{m,n}d(x_{_m},x_n)=0$;} (ii) {\rm Any two members of the family $\mathcal{A}$ (resp. $\mathcal{B}$) are right (resp. left) $d$-cofinal;}
{\rm (iii) The classes are maximal with respect to set inclusion}.
We call the member $\mathcal{A}$ (resp. ${\mathcal{B}})$ {\it first} (resp. {\it second}) class of $\xi$.
A $\kappa$-{\it Cauchy sequence} is a right $K$-Cauchy sequence which is member of the first class of a $\kappa$-cut (see \cite[Definition 11]{and1}).

A {\it directed set} is a nonempty set $A$ together with a reflexive and transitive binary relation $\leq$ (that is, a preorder), with the additional propositionerty that for any $x$ and $y$ in $A$ there must exist $z$ in $A$ with $x\leq z$ and $y\leq z$. 
Directed sets are a generalization of nonempty totally ordered sets. 
That is, all totally ordered sets are directed sets. 
A {\it net} in a topological space $X$ is a function $\delta: A\rightarrow X$, 
where $A$ is some directed set.  The point $\delta(a)$ is usually denoted $x_{_a}$
and the net is denoted by $(x_{_a})_{_{a\in A}}$. 
A function $\varphi: D\rightarrow A$ is {\it cofinal in }  $A$ if 
for each $a\in A$, there exists some $\mu\in D$ such that $a\leq \varphi(\mu)$. 
A {\it subnet} of a net $\delta: A\rightarrow X$ is the composition
$\delta\circ\varphi$, where $\varphi: D\rightarrow A$ is an increasing
cofinal function from a directed set $D$ to $A$. That is:
(i) $\varphi(\mu_{_1})\leq \varphi(\mu_{_2})$ whenever $\mu_{_1}\leq \mu_{_2}$ ($\varphi$ is increasing);
(ii) $\varphi$ is cofinal in $A$.
For each $\mu\in M$, the point $\delta\circ \varphi(\mu)$ is often written $x_{_{a_{_\mu}}}$, and we usually speak of ``the subnet $(x_{_{a_{_\mu}}})_{_{\mu\in M}}$ of $(x_{_a})_{_{a\in A}}$."
A net $(x_{_a})_{_{a\in A}}$
in quasi-uniform space $(X,\mathcal{U})$ is said to be {\it convergent} to $x\in X$ if for every $U\in \mathcal{U}$ there exists $
a_{_U}\in \mathbb{D}$ such that $(x,x_{_a})\in U$ for each $a\geq a_{_U}$.
The definition of net generalizes a key result about subsequences:
A net $(x_{_a})_{_{a\in A}}$ converges to $x$ if and only if every subnet of $(x_{_a})_{_{a\in A}}$ converges to $x$.

Let $X$ be a well-ordered set and let $R(x_{_\lambda})$, $\lambda\in \Lambda$ be a propositionosition with domain $X$.
The {\it Principle of Transfinite Induction} asserts that if $\displaystyle\bigcup_{_{\lambda<\mu}}R(x_{_\lambda})$ implies $R(x_{_\mu})$, for all $\mu \in \Lambda$, then in fact $R(x_{_\lambda})$ holds for all $\lambda\in\Lambda$.

Let $X_{_a}$ be a set, for each $a\in A$. The {\it Cartesian product} of the sets $X_{_a}$ is the set 

\begin{center}
$\displaystyle\prod_{a\in A}X_{_a}=\{x: A\rightarrow \displaystyle\bigcup_{a\in A}X_{_a}\vert \ x(a)\in X_{_a},$
for each $a\in A\}$.
\end{center}

The value of $x\in \displaystyle\prod_{a\in A}$ at $a$ is usually denoted $x_{_a}$, rather than $x(a)$, and $x_{_a}$ is referred to as the {\it ath coordinate} of $x$. The space $X_{_a}$ in the {\it ath factor space}.
For each $\gamma\in A$ the map $\pi_{_\gamma} : \displaystyle\prod_{a\in A}X_{_a}\rightarrow X_{_\gamma}$, defined by  $\pi_{_\gamma}(x)=x_{_\gamma}$,is called the {\it projection map} of $\displaystyle\prod_{a\in A}X_{_a}$ on $X_{_\gamma}$, or the $\gamma$th {\it projection map}. The Axiom of choice  ensure that the Cartesian product of a non-empty collection of non-empty sets is non-empty.

Let $(X_{_i},{\cal{U}}_{_i})_{_{i\in I}}$ be a family of quasi-uniform spaces. Let  $\displaystyle\prod_{i\in A}X_{_i}$ be the set-theoretic product of the family $(X_{_i})_{_{i\in I}}$ and let $\pi_{_j}: \displaystyle\prod_{i\in A}X_{_i}\longrightarrow X_{_j}$ ($j\in I$) be the projection onto $(X_{_j},{\cal{U}}_{_j})$.
Then the coarsest quasi-uniformity on $\displaystyle\prod_{i\in A}X_{_i}$ that makes all projections uniformly continuous is called the product quasi-uniformity. It induces the product topology of the topological spaces $(X_{_i},{\cal{U}}_{_i})$.

\section{The new completion}
Throughout the paper $(X,\mathcal{U})$ will be an arbitrary quasi-uniform space 
and $(X,d)$ will be an arbitrary quasi-pseudometric space, except the cases when it is explicitly stated that the space is $T_{_0}$.

Stoltenberg \cite{sto} has given the following definitions. The index $S$ in our symbolisms is devoted to Stoltenberg.

\begin{definition}\label{a0}{\rm (\cite[Definition 2.1]{sto}).
A net $(x_{_a})_{_{a\in A}}$ in a quasi-unifrom space $(X,\mathcal{U})$ is called 
{\it right} $\mathcal{U}_{_S}$-{\it Cauchy} ({\it left} $\mathcal{U}_{_S}$-{\it Cauchy}) if for each $U\in\mathcal{U}$ there is $a_{_U}\in A$ such that $(x_{_\beta},x_{_\alpha})\in U$ (resp. $(x_{_\alpha},x_{_\beta})\in U$)
whenever $\alpha\geq a_{_U}$, $\beta\geq  a_{_U}$, $a\ngeq \beta$ and $a, \beta\in A$.} \end{definition}

\begin{definition} {\rm (\cite[Definition 2.2]{sto}). 
A quasi-uniform space $(X,\mathcal{U})$ is $\mathcal{U}_{_S}$-{\it complete}
if and only if each right $\mathcal{U}_{_S}$-Cauchy net 
converges in $X$ with respect to $\tau_{_{\mathcal{U}}}$.}
\end{definition}

\begin{definition}\label{a1}{\rm 
A net $(x_{_a})_{_{a\in A}}$ in a quasi-pseudometric space $(X,d)$ is called {\it right} (resp. {\it left}) $d_{_S}$-{\it Cauchy}
if for each $\epsilon>0$ there is $a_{_\epsilon}\in A$ such that $d(x_{_\beta},x_{_\alpha})<\epsilon$ (resp. $d(x_{_\alpha},x_{_\beta})<\epsilon$) whenever $\alpha\geq a_{_\epsilon}$,
$\beta\geq  a_{_\epsilon}$, $a\ngeq \beta$ and $a, \beta\in A$.
Without loss of generality, we may suppose that for $\varepsilon^\prime\leq \varepsilon$, it is $a_{_{\varepsilon^{\prime}}}\geq a_{_\varepsilon}$ (resp. 
$\beta_{_{\varepsilon^{\prime}}}\geq \beta_{_\varepsilon}$).
We say that a quasi-pseudometric space $(X,d)$ is $d_{_S}$-{\it complete} if and only if every right $d_{_S}$-Cauchy net converges to a point in $X$ with respect to $\tau_{_d}$.
Similarly, we say that $(X,d)$ is $\mathcal{U}_{_d}$-{\it complete}
if and only if every $\mathcal{U}_{_d}$-Cauchy net converges to a point in $X$ with respect to $\tau_{_{\mathcal{U}_{_d}}}=\tau_{_d}$.}
\end{definition}

In case of sequences, Definition \ref{a1} coincides with the definition of the notions of right $K$-Cauchy sequence, left $K$-Cauchy sequence, right $K$-sequentially complete and left $K$-sequentially complete quasi-pseudometric space, respectively.

\begin{proposition}\label{a2}{\rm (see \cite[Page 229]{sto}). Let $(X,d)$ be a quasi-pseudometric space. A net $(x_{_a})_{_{a\in A}}$ 
in $(X,d)$ is a right $d_{_S}$-Cauchy net if and only if $(x_{_a})_{_{a\in A}}$ is a right $(\mathcal{U}_{_d})_{_S}$-Cauchy net in $(X,\mathcal{U}_d)$. Simillarly,
the space $(X,d)$ is $d_{_S}$-complete if and only if $(X,\mathcal{U}_{_d})$ is
$\mathcal{U}_{_d}$-complete.}
\end{proposition}

\begin{proof} By definition, we have that
each quasi-pseudometric $d$ on $X$ generates a quasi-uniformity $\mathcal{U}_{_d}$
with base $\{\{(x,y)\in X\times X\vert d(x,y)<\epsilon\}\vert\epsilon>0\}$. Therefore, the implication of the Proposition is an immediate consequence of the Definitions \ref{a0}-\ref{a1}.
\end{proof}

\begin{definition}\label{321} {\rm (\cite[Definition 1]{doi1}, \cite{doi}). Let $(X,\cal{U})$ (resp. $(X,d)$) be a quasi-uniform space (quasi-pseudometric space) and let $(x_{_a})_{_{a\in A}}$, $(y_{_\beta})_{_{\beta\in B}}$ be two nets in $(X,\cal{U})$ (resp. $(X,d)$). The net $(y_{_\beta})_{_{\beta\in B}}$ is 
called a {\it conet} of $(x_{_a})_{_{a\in A}}$, if for any $U\in \cal{U}$ (resp. $\epsilon > 0$) there are $a_{_U}\in A$ and $\beta_{_U}\in B$ (resp. $a_{_\epsilon}\in A$ and $\beta_{_\epsilon}\in B$) such that $(y_{_\beta},x_{_a})\in U$ (resp. $d(y_{_\beta},x_{_a})<\epsilon$)
whenever $a\geq a_{_U}$, $\beta\geq \beta_{_U}$ (resp. $a\geq a_{_\epsilon}$, $\beta\geq \beta_{_\epsilon}$) and $a, \beta \in A$.}
\end{definition}

\begin{definition}{\rm (\cite[Definition 3]{and1} ).
A net $(x_{_a})_{_{a\in A}}$ in a quasi-pseudometric space $(X,d)$ is called
{\it right} (resp. {\it left}) $d$-{\it cofinal to} a net $(x_{_\beta})_{_{\beta\in B}}$ on $X$, if for each $\epsilon>0$ there exists $a_{_\epsilon}\in A$ satisfying the following property: 
for each $a\geq a_{_\epsilon}$ there exists $\beta_{_a}\in B$ such that $d(x_{_\beta},x_{_{a}})<\epsilon$ 
(resp. $d(x_{_{a}},x_{_\beta})<\epsilon$) whenever $\beta\geq \beta_{_a}$. The nets $(x_{_a})_{_{a\in A}}$ and $(x_{_\beta})_{_{\beta\in B}}$
are {\it right} (resp. {\it left}) $d$-{\it cofinal} if $(x_{_a})_{_{a\in A}}$
is right (resp. left) $d$-cofinal to $(x_{_\beta})_{_{\beta\in B}}$ and vice versa.}
\end{definition}

\begin{definition}\label{ewq} {\rm Let $(X,d)$ be a quasi-pseudometric space. We call $\delta$-cut in $X$ an ordered pair $\xi=(\mathcal{A}_{_\xi},\mathcal{B}_{_\xi})$
of families of right $d_{_S}$-Cauchy nets and left $d_{_S}$-Cauchy nets
respectively, with the following properties:

(i) Any $(x_{_a})_{_{a\in A}}\in\mathcal{A}_{_\xi}$ has as conet any $(y_{_\beta})_{_{\beta\in B}}\in\mathcal{B}_{_\xi}$;

(ii) Any two members of the family $\mathcal{A}_{_\xi}$ (resp. $\mathcal{B}_{_\xi}$) are right (resp. left) $d$-cofinal.

(iii) The classes $\mathcal{A}_{_\xi}$ and $\mathcal{B}_{_\xi}$ are maximal with respect to set inclusion.}
\end{definition}

\par\noindent
\begin{definition}\label{a6} {\rm To every $x\in X$ we correspond the $d$-cut
$\phi(x)=(\mathcal{A}_{_{\phi(x)}},\mathcal{B}_{_{\phi(x)}})$ satisfying the requirements (i)-(iii) of Definition \ref{ewq} with the additional condition (iv): The net $(x)=x, x, x, ...$ itself, belongs to both of the classes.}
\end{definition}

The notion of $\delta$-cut of nets generalizes the notion
of $\kappa$-cut of sequences (see \cite[Definition 11]{and1}).

\begin{remark}\label{p12}{\rm In fact, in the previous definition the members of $\mathcal{A}_{_{\phi(x)}}$ converge to $x$ with respect to
$\tau_{_{d}}$ and the members of $\mathcal{B}_{_{\phi(x)}}$ converge to $x$ with respect to $\tau_{_{d^{-1}}}$.}
\end{remark}

Let $(X,d)$ be a quasi-pseudometric space and let $\widehat{X}$ denotes the set of all $d$-cuts in $(X,d)$. Throughtout the paper, for every $\xi\in \widehat{X}$,
$\mathcal{A}_{_{\xi}}, \mathcal{B}_{_{\xi}}$ denote the two classes of  $\xi$. 
In this case, we write $\xi=(\mathcal{A}_{_{\xi}},\mathcal{B}_{_{\xi}})$.

\begin{remark}\label{a600} {\rm If the space $(X,d)$ is $T_{_0}$, then
the function $\phi$ defined above is an injective function (one-to-one) of $X$ into $\widehat{X}$. Indeed, let $x, y\in X$ be such that $\phi(x)=\phi(y)$. Then, 
$(x), (y) \in \mathcal{A}_{_{\phi(x)}}\cap \mathcal{B}_{_{\phi(x)}}\cap \mathcal{A}_{_{\phi(y)}}\cap \mathcal{A}_{_{\phi(y)}}$. Thus,
$d^{\ast}(x,y)=0$ which implies that $x=y$.
Generally, each quasi-pseudometric space $(X,d)$ defines a $T_{_0}$ quasi-pseudometric space $(X^{\ast},d^{\ast})$.
Indeed, let $R$ be the equivalence relation in $X$ defined by: $xRy$ if and only if $d(x,y)=0=d(y,x)$. 
Let also $X^{\ast}=X/R=\{[x]\ \vert \ x\in X\}=\{\{z\in X \vert zRx\}\ \vert x\in X\}$ be the quotient space (the set of equivalence classes). It is convenient to introduce the function $\pi: X\rightarrow X^{\ast}$ defined by $\pi(x)=\{[y] \vert\ x\in [y]\}$.
Since each $x\in X$ is contained in exactly one equivalence class we have $\pi(x)=[x]$, that is, the function $x\rightarrow [x]$ is well defined.
In order to ensure that the projection map be continuous we put an obligation on 
the topology we assign to $X^{\ast}$: If a set $G$ is open in $X^{\ast}$ then $\pi^{-1}(G)$ is open in $X$. We define the {\it quotient topology} on $X^{\ast}$ by letting all sets $G$ that pass this test be admitted. On other words, a set $G$ is open in $X^{\ast}$ if and only if $\pi^{-1}(G)$ is open in $X$.
The quotient topology on $X^{\ast}$ is the finest topology on $X^{\ast}$ for which the projection map $\pi$ is continuous. Let $d^{\ast}: X^{\ast}\times X^{\ast}\rightarrow \mathbb{R}$ be a function defined by $d^{\ast}([x], [y])=d(x,y)$. Then, it is easy to check that $d^{\ast}$ is  
a $T_{_0}$ quasi-pseudometric in $X^{\ast}$ that yields the quotient topology in $X^{\ast}$.}
\end{remark}

\par\noindent
\begin{definition}\label{a3}{\rm Let $(X,d)$ be a quasi-pseudometric space.
We call $\delta$-{\it Cauchy net} any right $d_{_S}$-Cauchy net member of the first class of a $\delta$-cut.
The space $(X,d)$ is $\delta$-{\it complete} if and only if each $\delta$-Cauchy
net converges in $X$.}
\end{definition}

\begin{definition}{\rm A net $(x_{_a})_{_{a\in A}}$ in a quasi-uniform space $(X,
\mathcal{U})$ is called {\it right} (resp. {\it left}) $\mathcal{U}$-{\it cofinal to} a net $(x_{_\beta})_{_{\beta\in B}}$ 
on $X$, if for each $U\in \mathcal{U}$ there exists $a_{_U}\in A$ satisfying the following property: 
for each $a\geq a_{_U}$ there exists $\beta_{_a}\in B$ such that $(x_{_\beta},x_{_{a}})\in U$ (resp. $(x_{_{a}},x_{_\beta})\in U$)
whenever $\beta\geq \beta_{_a}$. The nets $(x_{_a})_{_{a\in A}}$ and
$(x_{_\beta})_{_{\beta\in B}}$ are {\it right} (resp. {\it left}) $\mathcal{U}$-{\it cofinal} if $(x_{_a})_{_{a\in A}}$
is right (resp. left) $\mathcal{U}$-cofinal to $(x_{_\beta})_{_{\beta\in B}}$ and vice versa.}
\end{definition}

\begin{definition}\label{ewq1} {\rm Let $(X,\mathcal{U})$ be a quasi-uniform space. We call $\mathcal{U}$-cut in $X$ an ordered pair $\xi=(\mathcal{A}_{_\xi},\mathcal{B}_{_\xi})$ of families of right $\mathcal{U}_{_S}$-Cauchy nets and left $\mathcal{U}_{_S}$-Cauchy nets
respectively, with the following properties:

(i) Any $(x_{_a})_{_{a\in A}}\in\mathcal{A}_{_\xi}$ has as conet any $(y_{_\beta})_{_{\beta\in B}}\in\mathcal{B}_{_\xi}$;

(ii) Any two members of the family $\mathcal{A}_{_\xi}$ (resp. $\mathcal{B}_{_\xi}$) are right (resp. left) $\mathcal{U}$-cofinal.

(iii) The classes $\mathcal{A}_{_\xi}$ and $\mathcal{B}_{_\xi}$ are maximal with respect to set inclusion.}
\end{definition}

\par\noindent
\begin{definition}\label{da6} {\rm To every $x\in X$ we correspond the $\mathcal{U}$-cut
$\phi(x)=(\mathcal{A}_{_{\phi(x)}},\mathcal{B}_{_{\phi(x)}})$ satisfying the requirements (i)-(iii) of Definition \ref{ewq}
with the additional condition (iv): 
The members of $\mathcal{A}_{_{\phi(x)}}$ converge to $x$ with respect to
$\tau_{_{\mathcal{U}}}$
and the members of $\mathcal{B}_{_{\phi(x)}}$ converge to $x$ with respect to
$\tau_{_{\mathcal{U}^{-1}}}$.}
\end{definition}

According to Definition \ref{da6},
the net $(x)=x, x, x, ...$ itself, belongs to both of the classes $\mathcal{A}_{_{\phi(x)}}$ and
$\mathcal{B}_{_{\phi(x)}}$.

\begin{definition}\label{a13}{\rm Let $(X,\mathcal{U})$ be a quasi-uniform space.
We call $\mathcal{U}$-{\it Cauchy net} any right $\mathcal{U}_{_S}$-Cauchy net
member of the first class of a $\mathcal{U}$-cut.
The space $(X,\mathcal{U})$ is $\mathcal{U}$-{\it complete}
if and only if each $\mathcal{U}$-Cauchy
net converges in $X$.}
\end{definition}

Two $\delta$-Cauchy (resp. $\mathcal{U}$-Cauchy) nets $(x_{_a})_{_{a\in A}}$ and $(x_{_\beta})_{_{\beta\in B}}$ in a
quasi-pseudometric space $(X,d)$ (resp. quasi-uniform space $(X,\mathcal{U})$) are called $\delta$-{\it equivalent} (resp. $\mathcal{U}$-{\it equivalent}) if every conet to 
$(x_{_a})_{_{a\in A}}$
is a conet to $(x_{_\beta})_{_{\beta\in B}}$
and vice versa.

It is easy to see that $\delta$-equivalence (resp. $\mathcal{U}$-equivalence) defines an equivalence relation on $(X,d)$ (resp. $(X,\mathcal{U})$).
Hence, $\mathcal{A}$ is the equivalence class of 
the $\delta$-Cauchy nets (resp. $\mathcal{U}$-Cauchy nets)
that are considered to be equivalent by this equivalence relation. 

In metric (resp. uniform) spaces the classes $\mathcal{A}$ and $\mathcal{B}$ coincide with 
a well known equivalent classes of Cauchy nets.

\begin{proposition}\label{a8} {\rm Let $(x_{_a})_{_{a\in A}}$ be a right
(resp. left) $d_{_S}$-Cauchy net in a quasi-pseudometric space
$(X,d)$ with a subnet $(x_{_{a_{_\gamma}}})_{_{\gamma\in \Gamma}}$. Then, 
$(x_{_a})_{_{a\in A}}$
and
$(x_{_{a_{_\gamma}}})_{_{\gamma\in \Gamma}}$ are right (resp. left) $d$-cofinal.}
\end{proposition}
\begin{proof}
Let $(x_{_a})_{_{a\in A}}$ be a right $d_{_S}$-Cauchy net in $(X,d)$ and $(x_{_{a_{_\gamma}}})_{_{\gamma\in \Gamma}}$ be a
subnet of it.
Let $\epsilon>0$ be given. Then, there exists $a_{_\epsilon}\in A$ such that for each $a, a^{\prime}\in A$
with
$a\geq a_{_\epsilon}$, $a^\prime\geq a_{_\epsilon}$ and $a^{\prime}\ngeq a$, we have $d(x_{_a},x_{_{a^{\prime}}})<\epsilon$.
Let $a_{_\gamma}>a_{_\epsilon}$ for some $\gamma\in\Gamma$ and $a_{_{a_{_{\gamma}}}}=a_{_\gamma}$. Then, for each $a> a_{_\gamma}$ ($a_{_\gamma}\ngeq a$) we have $d(x_{_a},x_{_{a_{_\gamma}}})<\epsilon$.
Hence, $(x_{_{a_{_\gamma}}})_{_{\gamma\in \Gamma}}$ is right $d$-cofinal to 
$(x_{_a})_{_{a\in A}}$.
On the other hand, let $a> a_{_\epsilon}$ for some $a\in A$. Let also 
$(a_{_\gamma})_{_a}=a_{_\gamma}$
for some $a_{_\gamma}>a$, $\gamma\in \Gamma$. Then, for each 
$\gamma^{\prime}>\gamma$
we have $d(x_{_{a_{_{\gamma^{\prime}}}}},x_{_a})<\epsilon$ which implies that 
$(x_{_a})_{_{a\in A}}$
is right $d$-cofinal to
$(x_{_{a_{_\gamma}}})_{_{\gamma\in \Gamma}}$. 
The case of the left $d_{_S}$-Cauchy net is simillar.
\end{proof}

\begin{proposition}\label{a9}{\rm In every quasi-pseudometric space $(X,d)$ two right 
$d$-cofinal nets have the same conets.}
\end{proposition}
\begin{proof}
Let $(x_{_a})_{_{a\in A}}$, $(x_{_\beta})_{_{\beta\in B}}$ be two right $d$-cofinal nets. 
Suppose that $(y_{_\sigma})_{_{\sigma\in\Sigma}}$ is a conet of  $(x_{_a})_{_{a\in A}}$. 
Fix  $\epsilon>0$.
Then there exist $\sigma_{_\epsilon}\in \Sigma$ and $a_{_\epsilon} \in A$ such that 
$d(y_{_\sigma}, x_{_a})<\displaystyle{\epsilon\over 2}$ for $\sigma\geq \sigma_{_\epsilon}$, $a\geq a_{_\epsilon}$. On the other hand, there
is $\beta_{_\epsilon}\in B$
with the following property:
for each $\beta\geq \beta_{_\epsilon}$ there exists
$a_{_\beta}\in A$ such that $d(x_{_a},x_{_\beta})<\displaystyle{\epsilon\over 2}$ whenever $a\geq a_{_\beta}$.
If $a^{\ast}=max\{a_{_\epsilon},a_{_\beta}\}$, then from $d(y_{_\sigma},x_{_{{a^{\ast}}}})<\displaystyle{\epsilon\over 2}$
and $d(x_{_{{a^{\ast}}}},x_{_\beta})<\displaystyle{\epsilon\over 2}$ we conclude that $d(y_{_\sigma}, x_{_\beta})<\displaystyle{\epsilon\over 2}+\displaystyle{\epsilon\over 2}=\epsilon$
for $\sigma\geq \sigma_{_\epsilon}$ and $\beta\geq \beta_{_\epsilon}$.
\end{proof}

\par
Similarly we can prove the following proposition.
\begin{proposition}\label{a10}{\rm In every quasi-pseudometric space $(X,d)$, two left 
$d$-cofinal 
nets are conets
of the same nets.}
\end{proposition}

The following two corollaries is an immediate consequence of Propositions \ref{a9} and \ref{a10}.  
\begin{corollary}\label{a11}{\rm In every quasi-pseudometric space $(X,d)$, two right (left) $d$-cofinal nets have the same
limit points for $\tau(d)$ (resp. $\tau(d^{-1})$).}
\end{corollary}

\begin{corollary}\label{a105}{\rm Let $(X,d)$ be a quasi-pseudometric space and let $\xi, \xi^{\prime}\in \widehat{X}$. 
Then, $\mathcal{A}_{_{\xi}}\cap \mathcal{A}_{_{\xi^{\prime}}}\neq \emptyset$ implies $\mathcal{A}_{_{\xi}}=\mathcal{A}_{_{\xi^{\prime}}}$.}
\end{corollary}

Let $(X,d)$ be a quasi-pseudometric space and
let $\widehat{X}$ be the set of all $d$-cuts in $(X,d)$. Throughtout the paper, for every $\xi\in \widehat{X}$,
$\mathcal{A}_{_{\xi}}, \mathcal{B}_{_{\xi}}$ denote the two classes of  $\xi$. 
In this case, we write $\xi=(\mathcal{A}_{_{\xi}},\mathcal{B}_{_{\xi}})$.

\begin{definition}\label{a21}{\rm Let $(X,d)$ be a quasi-pseudometric space. Suppose that $r$ is a nonnegative real number, $\xi^{\prime}, \xi^{\prime\prime}\in \widehat{X}$, $(x_a)_{{a\in A}}\in {\mathcal{A}}_{_{\xi^{\prime}}}$ and 
$(x_{_\gamma})_{_{\gamma\in \Gamma}}\in {\mathcal{B}}_{_{\xi^{\prime\prime}}}$. We put 
$\widehat{d}(\xi^{\prime},\xi^{\prime\prime})\leq r$ if:

(i) $\mathcal{A}_{_{\xi^{\prime}}}=\mathcal{A}_{_{\xi^{\prime\prime}}}$ or

(ii) For each $\epsilon>0$ there are 
$a_{\epsilon}\in A$, $\gamma_{_\varepsilon}\in \Gamma$ such that
\begin{equation}
d(x_a, x_{_\gamma})<r+\epsilon
 \label{eq:4535}
 \end{equation}
when $a\geq a_\varepsilon$ and $\gamma\geq \gamma_\varepsilon$.
If $\xi^{\prime}=\phi(x)$ for some $x\in X$, then the net $(x_a)_{a\in A}$ always coincides with the fixed net,
for which $x_a=x$ for all $a\in A$. 
Then, we let
\begin{equation}
\widehat{d}(\xi^{\prime},\xi^{\prime\prime})=inf\{r\ \vert\ \widehat{d}(\xi^{\prime},\xi^{\prime\prime})\leq r\}.
 \label{eq:4537}
 \end{equation}
}
\end{definition}

\begin{proposition}\label{daf}{\rm The truth of $\widehat{d}(\xi^{\prime},\xi^{\prime\prime})\leq r$ in Definition 
\ref{a21}(ii) depends only on $\xi^{\prime}$, $\xi^{\prime\prime}$, and $r$; it does not depend on the choice 
of the nets $(x_a)_{a\in A}$ and
$(x_{_\gamma})_{_{\gamma\in \Gamma}}$.}
\end{proposition}
\begin{proof} Let 
$(x_a)_{a\in A}$ and $(x_{_\beta})_{_{\beta\in B}}$
be two right $\mathcal{U}_{_S}$-Cauchy nets of the class 
$\mathcal{A}_{_{\xi^{\prime}}}$
and 
$(x_\gamma)_{\gamma\in \Gamma}$ and $(x_{_\delta})_{_{\delta\in \Delta}}$
be two right $\mathcal{U}_{_S}$-Cauchy nets of the class 
$\mathcal{A}_{_{\xi^{\prime\prime}}}$.
Then, for each $\varepsilon>0$ there are $a_{_\varepsilon}\in A$ and $\gamma_{_\varepsilon}\in \Gamma$
such that 
\begin{center}
$d(x_a,x_\gamma)<r+\varepsilon$
\end{center}
{\rm when $a\geq a_{_\varepsilon}$ and $\gamma\geq \gamma_{_\varepsilon}$.}
Choose 
an arbitrary positive number 
$\varepsilon^{\prime}$ so that 
$0<\varepsilon^{\prime}<\displaystyle{\varepsilon\over 3}$.
Then, there 
are $a_{_{\varepsilon^{\prime}}}\in A$ and $\gamma_{_{\varepsilon^{\prime}}}\in \Gamma$
such that 
\begin{center}
$d(x_a,x_\gamma)<r+\varepsilon^{\prime}$
\end{center}
{\rm when $a\geq a_{_{\varepsilon^{\prime}}}$ and $\gamma\geq \gamma_{_{\varepsilon^{\prime}}}$.}
Since 
$(x_a)_{a\in A}$
is left $d$-cofinal
to
$(x_{_\beta})_{_{\beta\in B}}$, 
there
is 
$a^{\prime}_{_\varepsilon}\in A$
satisfying the following property:
For each $a\geq a^{\prime}_{_\varepsilon}$ there exists
$\beta_{_a}\in B$ such that
\begin{center}
$d(x_{_\beta},x_a)<\displaystyle{\varepsilon\over 3}$ 
\end{center}
whenever $\beta\geq \beta_{_a}$. 
Fix an $a^{\prime}_{_\varepsilon}\geq a_{_{\varepsilon^{\prime}}}$ and let $a^{\prime}_{_\varepsilon}=a^{\ast}$. 
Then, we have
\begin{center}
$d(x_{_\beta},x_{_\gamma})\leq d(x_{_\beta},x^{\ast})+
d(x^{\ast},x_{_\gamma})<r+\varepsilon^{\prime}+\displaystyle{\varepsilon\over 3}$
\end{center}
whenever $\beta\geq \beta_{_{a^{\prime}_{_\varepsilon}}}$ and $\gamma\geq \gamma_{_{\varepsilon^{\prime}}}$.

Similarly, since $(x_{_\delta})_{_{\delta\in \Delta}}$
is left $d$-cofinal
to
$(x_{_\gamma})_{_{\gamma\in \Gamma}}$, 
there
is 
$\delta_{_\varepsilon}\in \Delta$
satisfying the following property:
For each $\delta\geq \delta_{_\varepsilon}$ there exists
$\gamma_{_\delta}\in \Gamma$ such that
\begin{center}
$d(x_{_\gamma},x_\delta)<\displaystyle{\varepsilon\over 3}$ 
\end{center}
whenever $\gamma\geq \gamma_{_\delta}$. 
Let $\gamma^{\ast}=\max\{\gamma_{_{\varepsilon^{\prime}}},\gamma_{_\delta}\}$. Then,
\begin{center}
$d(x_{_\beta},x_{_\delta})\leq d(x_{_\beta},x_{_{\gamma^{\ast}}})+d(x_{_{\gamma^{\ast}}},x_{_\delta})
<r+\varepsilon^{\prime}+\displaystyle{\varepsilon\over 3}+\displaystyle{\varepsilon\over 3}<r+\varepsilon$
\end{center}
whenever $\beta\geq \beta_{_{a^{\prime}_{_\varepsilon}}}$ and $\delta\geq \delta_{_{\varepsilon}}$.
\end{proof}

\begin{proposition}\label{a114}{\rm Let $\xi^{\prime},\xi^{\prime\prime}\in \widehat{X}$,
$(x_a)_{a\in A}\in \mathcal{A}_{_{\xi^{\prime}}}$ and $(x_{_\gamma})_{_{\gamma\in \Gamma}}\in \mathcal{A}_{_{\xi^{\prime\prime}}}$ and $\mathcal{A}_{_{\xi^{\prime}}}\neq \mathcal{A}_{_{\xi^{\prime\prime}}}$. Then,
\begin{center}
$\widehat{d}(\xi^{\prime},\xi^{\prime\prime})=\lim\limits_{a,\gamma}d(x_a,x_{_\gamma})$.
\end{center}}
\end{proposition}

\begin{proof} Let $\widehat{d}(\xi^{\prime},\xi^{\prime\prime})=r$. Then, for any $\varepsilon>0$ there are
$a_{_\varepsilon}\in A$ and $\gamma_{_\varepsilon}\in \Gamma$ such that
\begin{center}
$d(x_a,x_{_\gamma})<r+\varepsilon$
\end{center}
whenever $a\geq a_{_\varepsilon}$ and $\gamma\geq \gamma_{_\varepsilon}$. 
To prove that
$r-\varepsilon<d(x_a,x_{_\gamma})$ for 
$a\geq a_{_\varepsilon}$ and $\gamma\geq \gamma_{_\varepsilon}$,
suppose to the contrary there exist a subnet   
$(x_{_{a_{_\lambda}}})_{_{\lambda\in \Lambda}}$ of $(x_a)_{a\in A}$ and 
a subnet   
$(x_{_{\gamma_{_\mu}}})_{_{\mu\in M}}$ of $(x_{_\gamma})_{_{\gamma\in \Gamma}}$
such that for all $x_{_{a_{_\lambda}}}$, $x_{_{\gamma_{_\mu}}}$ there holds 
$d(x_{_{a_{_\lambda}}},x_{_{\gamma_{_\mu}}})\leq r-\varepsilon$.
Then, by Propositions \ref{a8}, \ref{a9}, \ref{daf} and Definition \ref{a21} 
we have that 
$\widehat{d}(\xi^{\prime},\xi^{\prime\prime})\leq r-\varepsilon$, a contradiction. Therefore, we have
$\widehat{d}(\xi^{\prime},\xi^{\prime\prime})=r=\lim\limits_{a,\gamma}d(x_a,x_{_\gamma})$.
\end{proof}

\begin{proposition}\label{zos}{\rm
$ \widehat{X}$ is quasi-pseudometric.}
\end{proposition}
\begin{proof} From Definition \ref{a21} it follows immediately that
$\widehat{d}(\xi,\xi)=0$ and $\widehat{d}(\xi,\xi^{\prime})\geq 0$ for all $\xi, \xi^{\prime}\in \widehat{X}$.
To prove that $\widehat{d}$ satisfies the triangle inequality, let $\xi,\xi^{\prime},\xi^{\prime\prime}\in \widehat{X}$.
We have four cases to consider:

(i) $\mathcal{A}_{_{\xi}}\neq \mathcal{A}_{_{\xi^{\prime}}}$ and
$\mathcal{A}_{_{\xi^{\prime}}}\neq \mathcal{A}_{_{\xi^{\prime\prime}}}$.
Suppose that
$\widehat{d}(\xi,\xi^{\prime})=r_1$, $\widehat{d}(\xi^{\prime},\xi^{\prime\prime})=r_2$, 
$(x_a)_{a\in A}\in \mathcal{A}_{_{\xi}}$, $(x_{_\beta})_{_{\beta\in B}}\in {\mathcal{A}}_{_{\xi^{\prime}}}$ and 
$(x_{_\gamma})_{_{\gamma\in \Gamma}}\in \mathcal{A}_{_{\xi^{\prime\prime}}}$.
Then, for any $\varepsilon>0$ there are $a_{_\varepsilon}\in A$ and $\beta_{_\varepsilon}\in B$
such that $d(x_a,x_{_\beta})<r_1+\displaystyle{\varepsilon\over 2}$ whenever $a\geq a_{_\varepsilon}$ and
$\beta\geq \beta_{_\varepsilon}$. Similarly, there are $\beta^{\prime}_{_\varepsilon}\in B$ 
and $\gamma_{_\varepsilon}\in \Gamma$ such that 
$d(x_\beta,x_{_\gamma})<r_2+\displaystyle{\varepsilon\over 2}$ whenever 
$\beta\geq \beta^{\prime}_{_\varepsilon}$ and
$\gamma\geq \gamma_{_\varepsilon}$. Let $B_{_0}=\max\{\beta_{_\varepsilon},\beta^{\prime}_{_\varepsilon}\}$.
Then, 
$d(x_a,x_{_\gamma})\leq d(x_{a},x_{_{\beta_{_{M_{_0}}}}})+d(x_{_{\beta_{_{M_{_0}}}}},x_{_\gamma})
<r_1+r_2+\varepsilon$. Hence, according to Definition \ref{a21}, we have 
\begin{center}
$\widehat{d}(\xi,\xi^{\prime\prime})\leq r_1+r_2=\widehat{d}(\xi,\xi^{\prime})+
\widehat{d}(\xi^{\prime},\xi^{\prime\prime})$.
\end{center}

(ii) $\mathcal{A}_{_{\xi}}\neq \mathcal{A}_{_{\xi^{\prime}}}$ and 
$\mathcal{A}_{_{\xi^{\prime}}}=\mathcal{A}_{_{\xi^{\prime\prime}}}$.
Suppose that $\widehat{d}(\xi,\xi^{\prime})=r$. Since 
$\mathcal{A}_{_{\xi^{\prime}}}=\mathcal{A}_{_{\xi^{\prime\prime}}}$,
Definition \ref{a21} implies that 
$\widehat{d}(\xi,\xi^{\prime})\leq r$ and $\widehat{d}(\xi^{\prime},\xi^{\prime\prime})=0$. 
Therefore,
\begin{center}
$\widehat{d}(\xi,\xi^{\prime\prime})\leq r=\widehat{d}(\xi,\xi^{\prime})+
\widehat{d}(\xi^{\prime},\xi^{\prime\prime})$.
\end{center}

(iii) $\mathcal{A}_{_{\xi}}=\mathcal{A}_{_{\xi^{\prime}}}$ and 
$\mathcal{A}_{_{\xi^{\prime}}}\neq\mathcal{A}_{_{\xi^{\prime\prime}}}$.

(iv) $\mathcal{A}_{_{\xi}}=\mathcal{A}_{_{\xi^{\prime}}}$ and 
$\mathcal{A}_{_{\xi^{\prime}}}=\mathcal{A}_{_{\xi^{\prime\prime}}}$. This case is trivial.
\end{proof}

\begin{proposition}\label{sos}{\rm Let $(X,d)$ be a quasi-pseudometric space. Then, for any $x, y\in X$ we have 
\begin{center}
$\widehat{d}(\phi(x),\phi(y))=d(x,y)$.
\end{center}
}
\end{proposition}
\begin{proof} We have that $(x)\in \mathcal{A}_{_{\phi(x)}}$
and $(y)\in \mathcal{A}_{_{\phi(y)}}$. If $\mathcal{A}_{_{\phi(x)}}=\mathcal{A}_{_{\phi(y)}}$,
then $\widehat{d}(\phi(x),\phi(y))=0=d(x,y)$. Otherwise,
Proposition \ref{a114} implies that $\widehat{d}(\phi(x),\phi(y))=d(x,y)$.
\end{proof}

\begin{proposition}\label{szos}{\rm
For any $\xi=( \mathcal{A}_{_{\xi}}, \mathcal{B}_{_{\xi}}) \in \widehat{X}$, 
$(x_a)_{a\in A}\in\mathcal{A}_{_{\xi}}$ implies that $\widehat{d}(\xi,\phi(x_{_a}))\rightarrow 0$ and 
$(x_{_\beta})_{_{\beta\in B}}\in\mathcal{B}_{_{\xi}}$ implies that $\widehat{d}(\phi(x_{_\beta}),\xi)\rightarrow 0$.
}
\end{proposition}
\begin{proof}
We now prove that if $(x_a)_{a\in A}\in\mathcal{A}_{_{\xi}}$ then $\phi(x_a)$ converges
to $\xi$. Fix an $\varepsilon>0$.
Since 
$(x_a)_{a\in A}$ is a right $d_{_S}$-Cauchy net, there exists $a_{\varepsilon}\in A$ such that 
\begin{equation}
d(x_a,x_{a^{\prime}})<\displaystyle{\varepsilon\over 3},\
 \ {\rm whenever}\
a\geq a_{\varepsilon},\ a^{\prime}\geq a_{\varepsilon}\
   {\rm and }\ a^{\prime}\ngeq a. 
\label{eq:321}
\end{equation}
Let $(x_{_\sigma})_{_{\sigma\in \Sigma}}\in \mathcal{A}_{_\xi}$. 
Then, since $(x_a)_{a\in A}$ and $(x_{_\sigma})_{_{\sigma\in \Sigma}}$ 
 right $d$-cofinal we have the property:
There exists ${\widehat{a}}_{_\varepsilon}\in A$ and $\sigma_{_{{\widehat{a}}_{_\varepsilon}}}\in \Sigma$,
such that
\begin{equation}
d(x_{_{\sigma}},x_a)<\displaystyle{\epsilon\over 3},\
 \ {\rm whenever}\ a\geq {\widehat{a}}_{_\varepsilon}\ and\ 
\sigma\geq \sigma_{_{{\widehat{a}}_{_\varepsilon}}}. 
\label{eq:3012}
\end{equation}
Let $a^{\prime\prime}\in A$ be such that $a^{\prime\prime}\geq a_{_\varepsilon}$ and 
$a^{\prime\prime}\geq {\widehat{a}}_{_\varepsilon}$ ($A$ is directed).
Fix an $a\geq a_{_\varepsilon}$ and a right $d_{_S}$-Cauchy net  $(x_{_\rho})_{_{\rho\in P}}\in\mathcal{A}_{_{\phi(x_a)}}$.
There exists $\rho_{_\varepsilon}\in P$
such that
\begin{equation}
d(x_a,x_\rho)<\displaystyle{\epsilon\over 3},\
 \ {\rm whenever}\ \rho\geq \rho_{_\varepsilon}. 
\label{eq:7612}
\end{equation}
We have two cases to consider - the case where $a\geq a^{\prime\prime}$
and the case where $a\ngeq a^{\prime\prime}$.

(i) Let $a\geq a^{\prime\prime}$. Then, since $a\geq a^{\prime\prime}\geq a_{_\varepsilon}$ and 
$a\geq a^{\prime\prime}\geq {\widehat{a}}_{_\varepsilon}$
we have
\begin{equation}
d(x_{_{\sigma}},x_{_\rho})\leq
d(x_{_{\sigma}},x_a)+
d(x_a,x_{_\rho})<\displaystyle{\epsilon\over 3}+\displaystyle{\epsilon\over 3}<\displaystyle{2\epsilon\over 3}<\epsilon,\
 \ {\rm whenever} \     \sigma\geq \sigma_{_{{\widehat{a}}_{_\varepsilon}}}      {\rm and}\ \rho\geq \rho_{\epsilon}. 
\label{eq:3072}
\end{equation}

(ii) Let $a\ngeq a^{\prime\prime}$. Then, since $a^{\prime\prime}\geq {\widehat{a}}_{_\varepsilon}$ and 
 $a, a^{\prime\prime}\geq a_{_\varepsilon}$ with $a\ngeq a^{\prime\prime}$ we have that
\begin{equation}
d(x_{_{\sigma}},x_{_\rho})\leq
d(x_{_{\sigma}},x_{_{a^{\prime\prime}}})+d(x_{_{a^{\prime\prime}}},x_a)+
d(x_a,x_{_\rho})<\displaystyle{\epsilon\over 3}+\displaystyle{\epsilon\over 3}+\displaystyle{\epsilon\over 3}=\epsilon,\
 \ {\rm whenever} \     \sigma\geq \sigma_{_{{\widehat{a}}_{_\varepsilon}}}      {\rm and}\ \rho\geq \rho_{\epsilon}. 
\label{eq:3073}
\end{equation}
Since $(x_{_\sigma})_{_{\sigma\in \Sigma}}\in \mathcal{A}_{_\xi}$, $(x_{_\rho})_{_{\rho\in P}}\in\mathcal{A}_{_{\phi(x_a)}}$,
then by (\ref{eq:3072}) and (\ref{eq:3073}) and Definition \ref{a21}, we conclude that
\begin{equation}
d(\xi,\phi(x_a))<\varepsilon\ \ {\rm for}\ \ a\geq a_{_\varepsilon}\ {\rm \ which\  implies\ that} \ 
\widehat{d}(\xi,\phi(x_{_a}))\rightarrow 0.
\label{eq:3074}
\end{equation}
Similarly we can prove that $\widehat{d}(\phi(x_{_\beta}),\xi)\rightarrow 0$.
\end{proof}

\begin{proposition}\label{a123}{\rm Let $(X,d)$ be a quasi-pseudometric space and let $(\xi_{_a})_{_{a\in A}}$ be a non-constant right $\widehat{d}_{_S}$-Cauchy net in 
$(\widehat{X},\widehat{d})$ without last element.  
Let also $a^{\ast}\in A$ be such that for each $a\geq a^{\ast}, a^{\prime}\geq
a^{\ast}$ and $a^{\prime}\ngeq a$ there holds $\widehat{d}(\xi_{_a},\xi_{_{a^{\prime}}})=0$.
Then,
there exists a right $d_{_S}$-Cauchy net $(t_{_\sigma})_{_{\sigma\in \Sigma}}$ in $(X,d)$ such that
$(\xi_{_a})_{_{a\in A}}$ and $(\phi(t_{_\sigma}))_{_{\sigma\in \Sigma}}$
are right $\widehat{d}$-cofinal
nets.}
\end{proposition}
\begin{proof}
($\alpha$) {\it The Construction of the right $d_{_S}$-Cauchy net 
$(t_{_\sigma})_{_{\sigma\in \Sigma}}$ in $(X,d)$.}
Let $(\xi_{_a})_{_{a\in A}}$ and $a^{\ast}$ be as in the proposition.
Since $\widehat{d}(\xi_{_a},\xi_{_{a^{\prime}}})=0$, for each $\epsilon>0$ we have that 
\begin{equation}
\widehat{d}(\xi_{_a},\xi_{_{a^{\prime}}})<\epsilon\ {\rm whenever}\ 
a\geq a^{\ast},\ a^{\prime}\geq a^{\ast} \ {\rm and}\ a^{\prime}\ngeq a.
\label{eq:1}
\end{equation}

Let $(\Gamma,\leq)$ be cofinal well-ordered subset of $(A,\leq)$
(the existence of such $\Gamma$ follows from the axiom of choice).
Consider the subnet $(\xi_{_{a_{_{\gamma}}}})_{_{\gamma\in\Gamma}}=(\xi_{_\gamma})_{_{\gamma\in\Gamma}}$ of $(\xi_{_{a}})_{_{a\in A}}$.
Then, $(\xi_{_{\gamma}})_{_{\gamma\in\Gamma}}$ is cofinal to 
$(\xi_{_a})_{_{a\in A}}$ and by (\ref{eq:1}) we have that
\begin{equation}
\widehat{d}(\xi_{_{a_{_\gamma}}},\xi_{_{a_{_{\gamma^{\prime}}}}})<\epsilon\ {\rm whenever}\ 
\gamma\geq \gamma^{\prime}\geq a^{\ast}.\label{eq:2}
\end{equation}
For any $\gamma\in \Gamma$, let $\xi_{_\gamma}=(\mathcal{A}_{_{\xi_{_\gamma}}},\mathcal{B}_{_{\xi_{_\gamma}}})$
and 
$(x_{_{\rho(k_{_\gamma})}})_{_{\rho(k_{_\gamma})\in P_{_{k_{_{\gamma}}}}}}$
be a member of $\mathcal{A}_{_{\xi_{_\gamma}}}$ where 
$k_{_\gamma}$ denote the different right $d_{_S}$-Cauchy nets of 
$\mathcal{A_{_{\xi_{_\gamma}}}}$
and
$\rho(k_{_\gamma})$
denote the
indices sets of $k_{_\gamma}$.
Let also $\rho_{_\epsilon}(k_{_\gamma})$ be the smallest index with the property
\begin{equation}
d(x_{_{\rho(k_{_\gamma})}},x_{_{\rho^{\prime}(k_{_\gamma})}})<\displaystyle{\epsilon\over 3}\
 \ {\rm whenever} \ \rho(k_{_\gamma})\geq  \rho_{_\epsilon}(k_{_\gamma}),\
\rho^{\prime}(k_{_\gamma})\geq
 \rho_{_\epsilon}(k_{_\gamma})\ {\rm and}\  \rho^{\prime}(k_{_\gamma})\ngeq \rho(k_{_\gamma}).
  \label{eq:3}
\end{equation}

For each $\gamma\in \Gamma$ fix an 
$k^{\ast}_{_{\gamma}}\in P_{_{k^{\ast}_{_{\gamma}}}}$.
Without loss of generality (see Remark \ref{p12}), we can assume that:
\begin{equation}
{\rm For \ each} \ \epsilon^\prime\leq \epsilon\  
{\rm we\ have\ that}\  \rho_{_\epsilon}(k^{\ast}_{_{\gamma}})\leq\rho_{_{\epsilon^\prime}}
(k^{\ast}_{_{\gamma}}).
 \label{eq:45}
 \end{equation}
Let 
$\gamma>\gamma^{\prime}$ and let $\epsilon>0$. Then,
$\widehat{d}(\xi_{_\gamma},\xi_{_{\gamma^{\prime}}})=0$ implies that
there exists
$\rho^{\epsilon}({k^{\ast}_{_{\gamma}}})\in P_{_{k^{\ast}_{_{\gamma}}}}$,
$\rho^{{\epsilon^{\prime}}}({k^{\ast}_{_{\gamma^{\prime}}}})\in P_{_{k^{\ast}_{_{\gamma^{\prime}}}}}$ such that
\begin{equation}
d(x_{_{\rho(k^{\ast}_{_\gamma})}},x_{_{\rho(k^{\ast}_{_{{\gamma}^{\prime}}})}})<\displaystyle{\epsilon\over 3}
 \ {\rm whenever}\ \rho(k^{\ast}_{_\gamma})\geq \rho^{\epsilon}({k^{\ast}_{_{\gamma}}})
 \ {\rm and}\ \rho(k^{\ast}_{_{\gamma^{\prime}}})\geq \rho^{{\epsilon^{\prime}}}({k^{\ast}_{_{\gamma^{\prime}}}}).
  \label{eq:4}
 \end{equation}

We advance to the construction of the demanded right $d_{_S}$-Cauchy net 
$(t_{_\lambda})_{_{\lambda\in \Lambda}}$ in $(X,d)$ by using transfinite induction on the well-ordered set $(\Gamma,\leq)$.
Let $\gamma_{_1}\geq\gamma_{_0}\geq a^{\ast}$ for some $\gamma_{_0},\gamma_{_1} \in \Gamma$ and let $\epsilon>0$. Then, 
\begin{equation}
d(x_{_{\rho(k^{\ast}_{_{\gamma_{_1}}})}},x_{_{\rho(k^{\ast}_{_{\gamma_{_0}}})}})<\displaystyle{\epsilon\over 3}
 \ {\rm whenever}\ \rho(k^{\ast}_{{\gamma_{_1}}})\geq \rho^{\epsilon}({
 k^{\ast}_{{\gamma_{_1}}}})
 \ {\rm and}\ 
 \rho(k^{\ast}_{{\gamma_{_0}}})\geq \rho^{\epsilon}({
 k^{\ast}_{{\gamma_{_0}}}}).
 \end{equation}
Let $\widetilde{\rho}_{\epsilon}({k^{\ast}}\!\!_{{\gamma_{_0}}})=\rho_{_{\epsilon}}({
 k^{\ast}_{{\gamma_{_0}}}})$ and 
$\widetilde{\rho}_{\epsilon}({k^{\ast}}\!\!_{{\gamma_{_1}}})=min\{
\rho_{_\epsilon}({
 k^{\ast}_{{\gamma_{_1}}}}),
\rho^{\epsilon}({
 k^{\ast}_{{\gamma_{_1}}}})\}$. Then, 
\begin{equation}
d(x_{_{\rho(k^{\ast}_{_{\gamma_{_1}}})}},x_{_{\rho(k^{\ast}_{_{\gamma_{_0}}})}})<\displaystyle{2\epsilon\over 3}<\epsilon
 \ {\rm whenever}\ \rho(k^{\ast}_{{\gamma_{_1}}})\geq 
 \widetilde{\rho}_{\epsilon}({k^{\ast}}\!\!_{{\gamma_{_1}}})
 \ {\rm and}\ 
 \rho(k^{\ast}_{{\gamma_{_0}}})\geq \widetilde{\rho}_{\epsilon}({k^{\ast}}\!\!_{{\gamma_{_0}}}). 
  \label{eq:6}
 \end{equation}
 We call equation (\ref{eq:6}), $T(1)$-Property.

Let $\gamma_{_2}\in \Gamma$ such that $\gamma_{_2}>\gamma_{_1}>\gamma_{_0}$.
Then, if $\widetilde{\rho}_{\epsilon}({k^{\ast}}\!\!_{{\gamma_{_2}}})=min\{
\rho_{_\epsilon}({
 k^{\ast}_{{\gamma_{_2}}}}),
\rho^{\epsilon}({
 k^{\ast}_{{\gamma_{_2}}}})\}$, as in (\ref{eq:6}), it is easy to check that

\begin{equation}
d(x_{_{\rho(k^{\ast}_{_{\gamma_{_2}}})}},x_{_{\rho(k^{\ast}_{_{\gamma_{_1}}})}})<\epsilon\ {\rm and}\ 
d(x_{_{\rho(k^{\ast}_{_{\gamma_{_2}}})}},x_{_{\rho(k^{\ast}_{_{\gamma_{_0}}})}})<\epsilon\\
 \ {\rm whenever}\ \rho(k^{\ast}_{{\gamma_{_2}}})\geq 
 \widetilde{\rho}_{\epsilon}({k^{\ast}}\!\!_{{\gamma_{_2}}}),\
 \rho(k^{\ast}_{{\gamma_{_1}}})\geq 
 \widetilde{\rho}_{\epsilon}({k^{\ast}}\!\!_{{\gamma_{_1}}})
 \ {\rm and}\ 
 \rho(k^{\ast}_{{\gamma_{_0}}})\geq \widetilde{\rho}_{\epsilon}({k^{\ast}}\!\!_{{\gamma_{_0}}}). 
  \label{eq:7}
   \end{equation}
 We call equation (\ref{eq:7}), $T(2)$-Property.

Let $\gamma\in \Gamma$ be an ordinal which has the $T(\gamma)$-Property, that is:
For each $\gamma^{\prime}<\gamma$ we have
\begin{equation}
d(x_{_{\rho(k^{\ast}_{{\gamma}})}}, x_{_{\rho(k^{\ast}_{{\gamma^{\prime}}})}})<\epsilon\ {\rm whenever}\ \rho(k^{\ast}_{{\gamma}})\geq 
 \widetilde{\rho}_{\epsilon}({k^{\ast}}\!\!_{{\gamma}})\ {\rm and}\
 \rho(k^{\ast}_{{\gamma^{\prime}}})\geq 
 \widetilde{\rho}_{\epsilon}({k^{\ast}}\!\!_{{\gamma^{\prime}}}). 
  \label{eq:8}
   \end{equation}

According to transfinite induction, if $T(\gamma)$ is true whenever $T(\gamma^{\prime})$ is true for all $\gamma^{\prime}< \gamma$, then $T(\gamma)$ is true for all $\gamma$.

We follow two steps:
\par
(i) {\it Successor case:} Prove that for any successor ordinal $\gamma+1$, 
$T(\gamma+1)$ follows from $T(\gamma)$.
\par
(ii) {\it Limit case:} Prove that for any limit ordinal $\gamma$, $T(\gamma)$ follows from [$T(\gamma^{\prime})$ for all $\gamma^{\prime}< \gamma]$.
\par\smallskip\noindent
{\it Step (i)}. Let $\gamma$ be a successor ordinal. Let also $\delta$ be an ordinal such that $\gamma<\delta$. If $\delta$ is a limit ordinal, then there exists an ordinal $\zeta^{\prime}$ such that 
$\gamma<\zeta^{\prime}<\delta$, otherwise, there exists an ordinal $\zeta$ such that $\gamma<\delta<\zeta$.
We only prove 
that $T(\gamma+1)$ follows from $T(\gamma)$ where $\gamma+1=\zeta$ (the case
$\gamma<\zeta^{\prime}<\delta$ is similar for $\gamma+1=\delta$).

By repeating the steps (\ref{eq:4})-(\ref{eq:7}) above for $\gamma^{\prime},\delta,\zeta$ where
$\gamma^{\prime}\leq\gamma$,
instead of 
$\gamma_{_0}, \gamma_{_1}, \gamma_{_2}$, considering them with
same layout, we conclude that
\begin{equation}
d(x_{_{\rho(k^{\ast}_{{\zeta}})}}, x_{_{\rho(k^{\ast}_{{\gamma^{\prime}}})}})<\epsilon\ {\rm whenever}\ \rho(k^{\ast}_{{\zeta}})\geq 
 \widetilde{\rho}_{\epsilon}({k^{\ast}}\!\!_{{\zeta}})\ {\rm and}\
 \rho(k^{\ast}_{{\gamma}})\geq 
 \widetilde{\rho}_{\epsilon}({k^{\ast}}\!\!_{{\gamma^{\prime}}}). 
  \label{eq:9}
   \end{equation}

Therefore, $T(\gamma+1)=T(\zeta)$ holds.

\par\smallskip\par\noindent
{\it Step (ii)}. Let $\gamma$ be a limit ordinal. As it is well-known, an 
ordinal $\gamma$ is a limit ordinal if and only if there is an ordinal less than $\gamma$, and whenever 
$\gamma^{\prime}$ is an ordinal less than $\gamma$, then there exists an ordinal $\gamma^{\prime\prime}$ such that 
$\gamma^{\prime}<\gamma^{\prime\prime}<\gamma$.
Therefore,
by repeating the steps (\ref{eq:4})-(\ref{eq:7}) above for $\gamma^{\prime},\gamma^{\prime\prime},\gamma$ 
instead of 
$\gamma_{_0}, \gamma_{_1}, \gamma_{_2}$, considering them with
same layout, we conclude that
\begin{equation}
d(x_{_{\rho(k^{\ast}_{{\gamma}})}}, x_{_{\rho(k^{\ast}_{{\gamma^{\prime}}})}})<\epsilon\ {\rm whenever}\ \rho(k^{\ast}_{{\gamma}})\geq 
 \widetilde{\rho}_{\epsilon}({k^{\ast}}\!\!_{{\gamma}})\ {\rm and}\
 \rho(k^{\ast}_{{\gamma^{\prime}}})\geq 
 \widetilde{\rho}_{\epsilon}({k^{\ast}}\!\!_{{\gamma^{\prime}}}). 
  \label{eq:9}
   \end{equation}
Therefore, $T(\gamma)$ holds. It follows that $T(\gamma)$ is true for all $\gamma\in \widehat{\Gamma}$,
where $\widehat{\Gamma}$ is cofinal to $\Gamma$.

Let $\widetilde{A}=\{(\gamma,\epsilon)\vert \ \gamma\in \widehat{\Gamma}, \ \epsilon>0\}$. 
We define an order on $\widetilde{A}$ as follows:

\begin{equation}
(\gamma^{\prime},\epsilon^{\prime})\leq (\gamma,\epsilon)\ {\rm if\ and\ only\ if}\ 
{\rm (i)}\ \gamma^{\prime}<\gamma \ {\rm or}\ {\rm (ii)}\ \gamma^{\prime}=\gamma\ 
{\rm and}\ \epsilon<\epsilon^{\prime}.
\label{eq:10}
\end{equation}

Clearly, $\widetilde{A}$ is directed with respect to $\leq$. Define the net
$(x_{_{\widetilde{\rho}_{\epsilon}({k^{\ast}}\!\!\!\!_{_{\gamma}})}})_{_{(\gamma,\epsilon)\in \widetilde{A}}}$.
By (\ref{eq:45}), for each $\epsilon^{\prime}<\epsilon$, we have $\widetilde{\rho}_{\epsilon}({k^{\ast}}\!\!\!_{{\gamma}})\leq \widetilde{\rho}_{\epsilon^{\prime}}({k^{\ast}}\!\!\!_{{\gamma}})$. 
Therefore, by (\ref{eq:45}) and $T(\gamma)$ property we conclude that
$(x_{_{\widetilde{\rho}_{\epsilon}({k^{\ast}}\!\!\!\!_{_{\gamma}})}})_{_{(\gamma,\epsilon)\in \widetilde{A}}}$
is a right $d_{_S}$-Cauchy net.
We now prove that the right $\widehat{d}_{_S}$-Cauchy nets
$(\xi_{_a})_{_{a\in A}}$ and
$(\phi(x_{_{\widetilde{\rho}_{\epsilon}({k^{\ast}}\!\!\!\!_{_{\gamma}})}}))_{_{(\gamma,\epsilon)\in \widetilde{A}}}$
are
left $\widehat{d}$-cofinal.
Since $(\xi_{_a})_{_{a\in A}}$ and $(\xi_{_\gamma})_{_{\gamma\in \Gamma}}$ are
left $\widehat{d}$-cofinal and 
$(\xi_{_\gamma})_{_{\gamma\in \Gamma}}$ and $(\xi_{_\gamma})_{_{\gamma\in \widehat{\Gamma}}}$
are
left $\widehat{d}$-cofinal, to prove that $(\xi_{_a})_{_{a\in A}}$ and 
$(\phi(x_{_{\widetilde{\rho}_{\epsilon}({k^{\ast}}\!\!\!\!_{_{\gamma}})}}))_{_{(\gamma,\epsilon)\in \widetilde{A}}}$
are
left $\widehat{d}$-cofinal, we have to prove that 
$(\xi_{_\gamma})_{_{\gamma\in \widehat{\Gamma}}}$ and 
$(\phi(x_{_{\widetilde{\rho}_{\epsilon}({k^{\ast}}\!\!\!\!_{_{\gamma}})}}))_{_{(\gamma,\epsilon)\in \widetilde{A}}}$
are
right $\widehat{d}$-cofinal.

In order to prove that 
$(\xi_{_\gamma})_{_{\gamma\in \widehat{\Gamma}}}$ and 
$(\phi(x_{_{\widetilde{\rho}_{\epsilon}({k^{\ast}}\!\!\!\!_{_{\gamma}})}}))_{_{(\gamma,\epsilon)\in \widetilde{A}}}$
are
right $\widehat{d}_{_S}$-cofinal, let $\epsilon^{\ast}>0$ and $\gamma_{_{\epsilon^{\ast}}}>a^{\ast}$.
By construction, we have 
\begin{equation}
d(x_{_{\widetilde{\rho}_{\epsilon}({k^{\ast}}\!\!\!\!_{_{\gamma}})}},
x_{_{\rho({k^{\ast}}\!\!\!\!_{_{\gamma_{_{\epsilon^{\ast}}}}})}})<\displaystyle{\epsilon\over 3},\
\ (\gamma,\epsilon)\in\widetilde{A}, \gamma>\gamma_{_{\epsilon^{\ast}}}, 
\ \epsilon>0\ \ {\rm and}\
 \rho(k^{\ast}_{{\gamma_{_{\epsilon^{\ast}}}}})\geq 
 \widetilde{\rho}_{\epsilon}({k^{\ast}}\!\!_{{\gamma_{_{\epsilon^{\ast}}}}}). 
\label{eq:1010}
\end{equation}
Since each
$(x_{_{\rho(k_{_{\gamma_{_{\epsilon^{\ast}}}}})}})_{_{
\rho(k_{_{\gamma_{_{\epsilon^{\ast}}}}})\in P_{_{k_{_{{\gamma_{_{\epsilon^{\ast}}}}}}}}}}\!\!\!\!\!\in
\mathcal{A}_{_{\xi_{_{\gamma_{_{\epsilon^{\ast}}}}}}}$ is right $d_{_S}$-cofinal to 
$(x_{_{\rho(k^{\ast}_{_{\gamma_{_{\epsilon^{\ast}}}}})}})_{_{
\rho(k^{\ast}_{_{\gamma_{_{\epsilon^{\ast}}}}})\in P_{_{k^{\ast}_{_{{\gamma_{_{\epsilon^{\ast}}}}}}}}}}\!\!\!\!\!\in
\mathcal{A}_{_{\xi_{_{\gamma_{_{\epsilon^{\ast}}}}}}}$
we have that
\begin{equation}
d(x_{_{\rho({k^{\ast}}\!\!\!\!_{_{\gamma_{_{\epsilon^{\ast}}}}})}},
x_{_{\rho({k}_{_{\gamma_{_{\epsilon^{\ast}}}}})}})<\displaystyle{\epsilon\over 3},\
 \ {\rm whenever}\
 \rho(k^{\ast}_{{\gamma_{_{\epsilon^{\ast}}}}})\geq  \rho_{_0}(k^{\ast}_{{\gamma_{_{\epsilon^{\ast}}}}})
  \ {\rm and }\ 
   \rho(k_{{\gamma_{_{\epsilon^{\ast}}}}})\geq  \rho_{_0}(k_{{\gamma_{_{\epsilon^{\ast}}}}}) . 
\label{eq:1011}
\end{equation}
It follows that 
\begin{equation}
d(x_{_{\widetilde{\rho}_{\epsilon}({k^{\ast}}\!\!\!\!_{_{\gamma}})}},
x_{_{\rho({k}_{_{\gamma_{_{\epsilon^{\ast}}}}})}})<\displaystyle{2\epsilon\over 3}< \epsilon,\
 \ {\rm whenever}\
 \rho(k^{\ast}_{{\gamma_{_{\epsilon^{\ast}}}}})\geq  \rho_{_0}(k^{\ast}_{{\gamma_{_{\epsilon^{\ast}}}}})
  \ {\rm and }\ 
   \rho(k_{{\gamma_{_{\epsilon^{\ast}}}}})\geq  \rho_{_0}(k_{{\gamma_{_{\epsilon^{\ast}}}}}) . 
\label{eq:1012}
\end{equation}
which implies that 
\begin{equation}
d(\phi(x_{_{\widetilde{\rho}_{\epsilon}({k^{\ast}}\!\!\!\!_{_{\gamma}})}}),
\xi_{_{\gamma_{_{\epsilon^{\ast}}})}}
< \epsilon,\
\ {\rm whenever}\
(\gamma, \epsilon)>(\gamma_{_{\epsilon^{\ast}}}, \epsilon^{\ast})
\ {\rm and }\ 
\rho(k_{{\gamma_{_{\epsilon^{\ast}}}}})\geq  \rho_{_0}(k_{{\gamma_{_{\epsilon^{\ast}}}}}) . 
\label{eq:1012}
\end{equation}
Therefore
$(\phi(x_{_{\widetilde{\rho}_{\epsilon}({k^{\ast}}\!\!\!\!_{_{\gamma}})}}))_{_{(\gamma,\epsilon)\in \widetilde{A}}}$
is
right $\widehat{d}_{_S}$-cofinal
to
$(\xi_{_\gamma})_{_{\gamma\in \widehat{\Gamma}}}$. On the other hand,
in view of Definition \ref{oly} and by the construction of
$x_{_{\widetilde{\rho}_{\epsilon}({k^{\ast}}\!\!\!\!_{_{\gamma}})}}$
immediately follows that 
$(\xi_{_\gamma})_{_{\gamma\in \widehat{\Gamma}}}$
right $\widehat{d}_{_S}$-cofinal
to
$(\phi(x_{_{\widetilde{\rho}_{\epsilon}({k^{\ast}}\!\!\!\!_{_{\gamma}})}}))_{_{(\gamma,\epsilon)\in \widetilde{A}}}$.
Therefore, the required net $(t_{_\sigma})_{_{\sigma\in\Sigma}}$ of the hypothesis is the 
net $(x_{_{\widetilde{\rho}_{\epsilon}({k^{\ast}}\!\!\!\!_{_{\gamma}})}})_{_{(\gamma,\epsilon)\in \widetilde{A}}}$.
\end{proof}

If the cardinality of index set $A$ in Proposition \ref{a123}
equals to the cardinality
$\aleph _{0}$ of the set of all natural numbers, then we have the following corollary (see also \cite[Proposition 25]{and1}).

\begin{corollary}\label{124}{\rm Let $(X,d)$ be a quasi-pseudometric space and let 
$(\xi_{_n})_{_{n\in \mathbb{N}}}$ be a non-constant right 
$K$-Cauchy sequence in 
$(\widehat{X},\widehat{d})$ without last element. 
Then,
there exists a right $K$-Cauchy sequence $(t_{_\nu})_{_{\nu\in \mathbb{N}}}$ in $(X,d)$ such that 
the sequences
$(\xi_{_n})_{_{n\in \mathbb{N}}}$ and $(\phi(t_{_\nu}))_{_{\nu\in \mathbb{N}}}$
are right $\widehat{d}$-cofinal
sequences.}
\end{corollary}
\begin{proof}
Let 
 let 
$(\xi_{_n})_{_{n\in \mathbb{N}}}$ be a non-constant right 
$K$-Cauchy sequence in 
$(\widehat{X},\widehat{d})$ without last element and let
$(x_{_{\widetilde{\rho}_{\epsilon}({k^{\ast}}\!\!\!\!_{_{\gamma}})}})_{_{(\gamma,\epsilon)\in \widetilde{A}}}$
be as in Proposition \ref{a123}. Since 
$(\xi_{_n})_{_{n\in \mathbb{N}}}$ is a sequence we have that $\Gamma\subseteq \mathbb{N}$ and
$\widetilde{A}\subseteq \mathbb{N}$. In this case we put
$(x_{_{\widetilde{\rho}_{\epsilon}({k^{\ast}}\!\!\!\!_{_{\gamma}})}})_{_{(\gamma,\epsilon)\in \widetilde{A}}}=
(x_{_{\widetilde{\rho}_{\epsilon}({k^{\ast}}\!\!\!\!_{_{{\gamma}(n)}})}})_{_{(\gamma(n),\epsilon)\in \mathbb{N}}}
$.
We define an order on $\mathbb{N}$ as follows:
\begin{equation}
(\gamma(n^{\prime}),\epsilon^{\prime})\leq (\gamma(n),\epsilon)\ {\rm if\ and\ only\ if}\ 
{\rm (i)}\ n^{\prime}<n \ {\rm or}\ {\rm (ii)}\ n^{\prime}=n\ 
{\rm and}\ \epsilon<\epsilon^{\prime}.
\label{eq:10}
\end{equation}
Clearly, $\widetilde{\mathbb{N}}=\{(\gamma(n),\epsilon)\vert \gamma(n)\in \mathbb{N}, \epsilon>0\}\subseteq \mathbb{N}$ 
is 
a linearly ordered set
with respect to $\leq$. 
Define the sequence 
$(x_{_{\widetilde{\rho}_{\epsilon}({k^{\ast}}\!\!\!\!_{_{{\gamma}(n)}})}})_{_{(\gamma(n),\epsilon)\in \mathbb{N}}}
$.
By Proposition \ref{a123} we have that
$(\xi_{_n})_{_{n\in \mathbb{N}}}$
and 
$(\phi(x_{_{\widetilde{\rho}_{\epsilon}({k^{\ast}}\!\!\!\!_{_{{\gamma}(n)}})}}))_{_{(\gamma(n),\epsilon)\in \mathbb{N}}}$
are
right $\widehat{d}$-cofinal sequences.
\end{proof}

\begin{proposition}\label{125}{\rm Let $(X,d)$ be a quasi-pseudometric space and let $(\xi_{_a})_{_{a\in A}}$ be a non-constant right $\widehat{d}_{_S}$-Cauchy net in 
$(\widehat{X},\widehat{d})$ without last element.
Then,
there exists a right $d_{_S}$-Cauchy net $(t_{_\sigma})_{_{\sigma\in \Sigma}}$ in $(X,d)$ such that the nets
$(\xi_{_a})_{_{a\in A}}$ and $(\phi(t_{_\sigma}))_{_{\sigma\in \Sigma}}$
are right $\widehat{d}$-cofinal
nets.}
\end{proposition}

\begin{proof} Let $(\widehat{X}, \widehat{d})$ and $(\xi_{_a})_{_{a\in A}}$ be as in the assumptions of the Proposition.
Without loss of generality we can assume that $a\neq a^{\prime}$ implies that $\mathcal{A}_{_{\xi_{_a}}}\neq 
\mathcal{A}_{_{\xi_{_{a^{\prime}}}}}$.
Pick $a_{_0}\in A$ such that $\widehat{d}(\xi_{_a},\xi_{_{{a^{\prime}}}})\leq 1$ whenever 
$a\geq a_{_0}$, $a^{\prime}\geq a_{_0}$ and 
$a^{\prime}\ngeq a$.
Given $a_{_n}\in A$ such that $\widehat{d}(\xi_{_a},\xi_{_{{a^{\prime}}}})\leq \displaystyle{1\over {2^n}}$
whenever $a\geq a^{\prime}\geq a_{_n}$, choose 
$a_{_{n+1}}\in A$ such that $a_{_{n+1}}\geq a_{_{n}}$ and
$\widehat{d}(\xi_{_a},\xi_{_{{a^{\prime}}}})\leq \displaystyle{1\over {2^{n+1}}}$
whenever $a\geq a^{\prime}\geq a_{_{n+1}}$. 
Then, the sequence $(\xi_{_{a_{_n}}})_{_{n\in\mathbb{N}}}$
is a right $K$-Cauchy sequence in $(\widehat{X}, \widehat{d})$.
We have two cases to consider; (i) There exists $a^{\ast}\in A$ such that 
$a^{\ast}>a_{_n}$ for each $n\in\mathbb{N}$; (ii) For each $a\in A$ there exists $n\in\mathbb{N}$ and 
$a_{_n}\in\mathbb{N}$ such that
$a>\!\!\!\!\!\!/\ a_{_n}$.  
In case (i), for each $a, a^{\prime}\geq a^{\ast}\geq a_{_n}, n\in\mathbb{N}$ and $a^{\prime}\ngeq a$,
we have $\widehat{d}(\xi_{_a},\xi_{_{a^{\prime}}})\leq \displaystyle{1\over {2^n}}$ for all $n\in \mathbb{N}$. Hence,
$\widehat{d}(\xi_{_a},\xi_{_{a^{\prime}}})=0$. Thus, Proposition \ref{a123} ensures
the existence of a
right $d_{_S}$-Cauchy net $(t_{_\sigma})_{_{\sigma\in \Sigma}}$ in $(X,d)$ such that the nets
$(\xi_{_a})_{_{a\in A}}$ and $(\phi(t_{_\sigma}))_{_{\sigma\in \Sigma}}$
are right $\widehat{d}$-cofinal
nets.
In case (ii), since $(\xi_{_{a_{_n}}})_{_{n\in \mathbb{N}}}$ is a subsequence of $(\xi_{_a})_{_{a\in A}}$
Proposition \ref{a8} implies that
$(\xi_{_{a_{_n}}})_{_{n\in \mathbb{N}}}$ and $(\xi_{_a})_{_{a\in A}}$
are right $\widehat{d}$-cofinal.
On the other hand, Corollary \ref{124} implies that 
there is right $K$-Cauchy sequence $(t_{_\nu})_{_{\nu\in \mathbb{N}}}$ in $(X,d)$ such that 
the sequences
$(\xi_{_{a_{_n}}})_{_{n\in \mathbb{N}}}$ and $(\phi(t_{_\nu}))_{_{\nu\in \mathbb{N}}}$
are right $\widehat{d}$-cofinal
nets. It follows that $(\xi_{_{a}})_{_{a\in A}}$ and
$(\phi(t_{_\nu}))_{_{\nu\in \mathbb{N}}}$ are right $\widehat{d}$-cofinal nets.
\end{proof}

\begin{corollary}\label{pop}{\rm A quasi-metric space $(X, d)$ is $\delta$-complete if and only if
every right $K$-Cauchy sequence converges to a point of $(X,d)$.}
\end{corollary}
\begin{proof} If a quasi-metric space is $\delta$-complete, then each $\delta$-Cauchy net converges in $X$.
Therefore, each $\delta$-Cauchy sequence converges in $X$.

Conversely, suppose that $(X,d)$ is a quasi-metric space in which every $\delta$-Cauchy sequence converges to a point of $X$.
Let $(x_a)_{a\in A}$ be a $\delta$-Cauchy net in $X$. Then, by Proposition \ref{125} 
we have two cases to consider: ($\mathfrak{a}$) $d(x_a,x_{a^{\prime}})=0$, where $a, a^{\prime}\geq a_{_0}$ for some $a_{_0}\in A$ and $a^{\prime}\ngeq a$; ($\mathfrak{b}$) There exists a subsequence $(x_{a_{_n}})_{n\in\mathbb{N}}$ of $(x_a)_{a\in A}$ such that $(x_a)_{a\in A}$ and $(x_{a_{_n}})_{n\in\mathbb{N}}$ are right $d$-cofinal.
In case ($\mathfrak{a}$) we have $x_a=x_{a_{_0}}$ for all $a\in A$ with $a_{_0}\ngeq a$. Therefore, 
$(x_a)_{a\in A}$ converges to $x_{a_{_0}}$. In case ($\mathfrak{b}$), $(x_{a_{_n}})_{n\in\mathbb{N}}$ converges 
to a point 
$l\in X$. By Corollary \ref{a11} we have that $(x_a)_{a\in A}$ converges to $l$ as well.
\end{proof}

By using the dual version of Proposition \ref{a123}, Corollary \ref{124} and Proposition \ref{125}
for left $d_{_S}$-Cauchy nets we have
the following proposition.

\begin{proposition}\label{pan}{\rm Let $(X,d)$ be a quasi-pseudometric space and let $(\eta_{_\beta})_{_{\beta\in B}}$ be a non-constant left $\widehat{d}_{_S}$-Cauchy net in 
$(\widehat{X},\widehat{d})$ without last element. 
Then,
there exists a left $d_{_S}$-Cauchy net $(t_{_\rho})_{_{\rho\in P}}$ in $(X,d)$ such that the nets
$(\eta_{_\beta})_{_{\beta\in B}}$ and $(\phi(t_{_\rho}))_{_{\rho\in P}}$
are left $\widehat{d}$-cofinal
nets.}
\end{proposition}

\par\noindent
\begin{theorem}\label{a19}{\rm Every quasi-pseudometric space $(X,d)$ has a $\delta$-completion.}
\end{theorem}
\begin{proof} Let 
$(\xi_{_a})_{_{a\in A}}$ be a $\delta$-Cauchy net in the space 
$(\widehat{X},\widehat{d})$.
Then, by definition \ref{a3},
there exists a $\delta$-cut $\widehat{\xi}\in \widehat{X}$ such that 
$(\xi_{_a})_{_{a\in A}}\in \mathcal{A}_{_{\widehat{\xi}}}$.
Let
\begin{equation}
\widehat{\xi}=(\mathcal{A}_{_{\widehat{\xi}}},\mathcal{B}_{_{\widehat{\xi}}})\  {\rm where}
\mathcal{A}_{_{\widehat{\xi}}}=\{(\xi^{^i}_{_k})_{_{k\in K_{_i}}}\vert i\in I\}\  {\rm and}\ \
{\mathcal{B}}_{_{\widehat{\xi}}}=\{(\eta^{^j}_{_\lambda})_{_{\lambda\in \Lambda_{_j}}}\vert j\in J\}.
\label{eq:1022}
\end{equation}

We prove that
there exists a $\delta$-cut $\xi$ in $(X,d)$ such that 
$(\xi_{_a})_{_{a\in A}}$
converges to $\xi$. 
\par\noindent
We define $\xi=(\mathcal{A}_{_\xi},\mathcal{B}_{_\xi})$, where
\begin{equation}
\mathcal{A}_{_{\xi}}=\{(x_{_\sigma})_{_{\sigma\in \Sigma}}\ \vert \ (x_{_\sigma})_{_{\sigma\in \Sigma}} 
 {\rm is\ a \
right}\ {\rm d_{_S}-Cauchy\ net\ in\ (X,d)\ such \ that \
(\phi(x_{_\sigma})_{_\sigma\in \Sigma}\in \mathcal{A}_{_{\widehat{\xi}}}}\}
\label{eq:1032}
\end{equation}
and
\begin{equation}
\mathcal{B}_{_{\xi}}=\{(y_{_\rho})_{_{\rho\in P}}\ \vert \ (y_{_\rho})_{_{\rho\in P}} 
 {\rm is\ a \
left}\ {\rm d_{_S}-Cauchy\ net\ in\ (X,d)\ such \ that \
(\phi(y_{_\rho})_{_{\rho\in P}}\in \mathcal{B}_{_{\widehat{\xi}}}}\}.
\label{eq:1033}
\end{equation}
By Propositions \ref{125} and \ref{pan} the classes $\mathcal{A}_{_{\xi}}$ and $\mathcal{B}_{_{\xi}}$ are non-void.
We first verify that $\xi=(\mathcal{A}_{_{\xi}},{\mathcal{B}}_{_{\xi}})$ constitute a $\delta$-cut in $(X,d)$. For this we
need to show that the pair $(\mathcal{A}_{_{\xi}},{\mathcal{B}}_{_{\xi}})$ satisfies the conditions of Definition \ref{ewq}.
We first prove 
the validity of Condition (i) of Definition \ref{ewq}.
Let
$(x_{_\sigma})_{_\sigma\in \Sigma}\in \mathcal{A}_{_{\xi}}$ and
$(y_{_\rho})_{_\rho\in P}\in {\mathcal{B}}_{_{\xi}}$.
Then, 
by construction of $\mathcal{A}_{_{\xi}}$ and $\mathcal{B}_{_{\xi}}$, we have $\lim\limits_{\rho,\sigma}
\widehat{d}(\phi(y_{_\rho}),\phi(x_{_{\sigma}}))=0$.
Hence, Proposition \ref{sos} implies that $\lim\limits_{\sigma,\rho}
d(y_{_\sigma},x_{_{\rho}})=0$. 
To prove that $\xi$ satisfies the second condition of Definition \ref{ewq}, let 
$(x_{_\sigma}){_{_{\sigma\in \Sigma}}}$,
$(x_{_{\rho}}){_{_{\rho\in P}}}$ be two right $d_{_S}$-Cauchy nets
of 
$\mathcal{A}_{_{\xi}}$. Since
$(\phi(x_{_\sigma}))_{_{\sigma\in \Sigma}}$, 
$(\phi(x_{_{\rho}}))_{_{{\rho}\in P}}$
belong to
$\mathcal{A}_{_{\widehat{\xi}}}$, it follows by the definition of $\xi$
that they are right 
$\widehat{d}$-cofinal.
Hence, 
Proposition \ref{sos} 
implies that
$(x_{_\sigma})_{_{\sigma\in \Sigma}}$ and
$(x_{_\rho})_{_{\rho\in P}}$
are right $d$-cofinal.
Finally, condition (iii)
of Definition \ref{ewq} is an immediate consequence of 
the maximality of 
$\mathcal{A}_{_{\xi}}$ and ${\mathcal{B}}_{_{\xi}}$, respectively.
\par
Similarly we can prove 
conditions (ii) and (iii) in the case of left $d_{_S}$-Cauchy nets, members of $\mathcal{B}_{_\xi}$.
\par
\smallskip
\par\noindent
We now prove that 
$(\xi_{_a})_{_{a\in  A}}$ converges to $\xi$.
Indeed, according to Proposition \ref{125}
there exists a right $d_{_S}$-Cauchy net 
$(x_{_\sigma})_{_{\sigma\in \Sigma}}$ in $(X,d)$ such that the nets 
$(\xi_{_a})_{_{a\in  A}}$ and $(\phi(x_{_\sigma}))_{_{\sigma\in \Sigma}}$
are right $\widehat{d}$-cofinal. By construction of $\xi$ we have that
$(x_{_\sigma})_{_{\sigma\in \Sigma}}\in \mathcal{A}_{_\xi}$.
By Proposition \ref{szos} we have that
$\phi(x_{_{\sigma}})\longrightarrow \xi$. Since 
$(\phi(x_{_{\sigma}}))_{_{\sigma\in \Sigma}}$ and $(\xi_{_a})_{_{a\in A}}$
are right $\widehat{d}$-cofinl, Proposition \ref{a11} implies that  
$\xi_{_a}\longrightarrow \xi$. 
It follows that $(X,d)$ is $\delta$-complete.
\end{proof}

\begin{proposition}\ {\rm A quasi-pseudometric space $(X,d)$ is $\delta$-complete if and only if
$(X,\mathcal{U}_{_d})$ is $\mathcal{U}_{_d}$-complete.}
\end{proposition}
\begin{proof} It is an immediate consequence of Proposition \ref{a2}.
\end{proof}

\begin{proposition}\label{212}{\rm Let $(X,d)$ be a quasi-pseudometric space and let $\mathcal{U}_{_d}$
be the quasi-uniformity induced on $X$ by $d$. Then, the space
$(X,\mathcal{U}_{_d})$ is $\mathcal{U}_{_d}$-complete if and only if the space
$(X,d)$ is $\delta$-complete.
}
\end{proposition}
\begin{proof} The result is an immediate consequence 
of Proposition \ref{a2} and the
fact that
the familly
$U_{\epsilon}=\{(x,y)\in X\times X\vert d(x,y)<\epsilon\ \vert,\ \epsilon>0\}$
is a base for $\mathcal{U}_{_d}$.
\end{proof}

\begin{lemma}\label{512}{\rm Let $(X,\mathcal{U})$ and $(Y,\mathcal{V})$ be arbitrary quasi-uniform spaces
and let $f: (X,\mathcal{U})\longrightarrow (X,\mathcal{V})$ be a quasi-uniformly continuous mapping.
If $\xi=(\mathcal{A}_{_\xi},\mathcal{B}_{_\xi})$ 
where $\mathcal{A}_{_\xi}=\{ (x^i_{_a})_{_{a\in A}} \vert i\in I\}$ and
$\mathcal{B}_{_\xi}=\{ (y^j_{_\beta})_{_{\beta\in B}} \vert j\in J\}$
is a $\mathcal{U}$-cut in $(X,\mathcal{U})$, then
$f(\xi)=(f(\mathcal{A}_{_\xi}),f(\mathcal{B}_{_\xi}))$ 
where $f(\mathcal{A}_{_\xi})=\{ (f(x^i_{_a}))_{_{a\in A}} \vert i\in I\}$ and
$f(\mathcal{B}_{_\xi})=\{ (f(y^j_{_\beta}))_{_{\beta\in B}} \vert j\in J\}$ is a 
a $\mathcal{V}$-cut in $(X,\mathcal{V})$.
}
\end{lemma}
\begin{proof} Since $f$ is quasi-uniform continuous, if $(x_{_a})_{_{a\in A}}$ is a right $\mathcal{U}_{_S}$-Cauchy
net in $(X,\mathcal{U})$ with a left $\mathcal{U}_{_S}$-Cauchy net $(y_{_\beta})_{_{\beta\in B}}$
as conet, then
 $(f(x_{_a}))_{_{a\in A}}$ is a right $\mathcal{V}_{_S}$-Cauchy
net in $(Y,\mathcal{V})$ with $(f(y_{_\beta}))_{_{\beta\in B}}$
a left $\mathcal{V}_{_S}$-Cauchy conet.
The rest is obvious.
\end{proof}

The following proposition is evident.
\begin{proposition}\label{412}{\rm A closed subspace of a $\mathcal{U}$-complete quasi-uniform space $(X,\mathcal{U})$ is $\mathcal{U}$-complete.}
\end{proposition}

\begin{proposition}\label{213}{\rm If $\{(X_{_i},\mathcal{U}_{_i})\ \vert \ i\in I\}$ is a family of 
${\mathcal{U}}_{_i}$-complete quasi-uniform spaces, then the product space 
$(\widetilde{X},\widetilde{\mathcal{U}})=(\displaystyle\prod_{_i\in I}X_{_i},\displaystyle\prod_{_i\in I}\mathcal{U}_{_i})$ is
$\widetilde{\mathcal{U}}$-complete.}
\end{proposition}
\begin{proof} Let $(\Xi_{_a})_{_{a\in A}}$ be a 
$\widetilde{\mathcal{U}}$-Cauchy net in $(\widetilde{X},\widetilde{\mathcal{U}})$.
Then, there exists a a $\widetilde{\mathcal{U}}$-cut $\Xi=(\mathcal{A}_{_{\Xi}},\mathcal{B}_{_{\Xi}})$
in $(\widetilde{X},\widetilde{\mathcal{U}})$ such that $(\Xi_{_a})_{_{a\in A}}\in \mathcal{A}_{_{\Xi}}$.
Let $(\Xi_{_a})_{_{a\in A}}\in \mathcal{A}_{_{\Xi}}$ 
and $(H_{_\beta})_{_{\beta\in B}}\in \mathcal{B}_{_{\Xi}}$.
Let $\Xi^a=\{\Xi^i_{_a}\vert i\in I\}$ for each $a\in A$ and 
$H^\beta=\{H^j_{_\beta}\vert j\in J\}$ for each $\beta\in B$.
Fix an $i\in I$ and choose a $U_i\in \mathcal{U}_i$. Given $U_i\times \displaystyle\prod_{j\neq i}(X_j\times X_j)$ there are $a_{_0}\in A$ and $\beta_{_0}\in B$ such that: (1) $(x^a_{_i},x^{a^{\prime}}_{_i})\in U$
whenever $a\geq a_{_0}$, $a^{\prime}\geq a_{_0}$ and $a^{\prime}\ngeq a$; 
(2) $(y^{\beta^{\prime}}_{_j},y^\beta_{_j})\in U$
whenever $\beta\geq \beta_{_0}$, $\beta^{\prime}\geq \beta_{_0}$ and $\beta^{\prime}\ngeq \beta$
and (3) $(y^\beta_{_j},x^a_{_i})\in U$ whenever $a\geq a_{_0}$ and $\beta\geq \beta_{_0}$.
Thus, $\Xi_i=(\mathcal{A}_{_{{\Xi}_i}},\mathcal{B}_{_{{\Xi}_i}})$ is a $\mathcal{U}_i$-cut
in $(X_{_i},\mathcal{U}_{_i})$ and $(x^a_i)_{_{a\in A}}$ is a $\mathcal{U}_i$-Cauchy net in
$(X_{_i},\mathcal{U}_{_i})$.
Thus, $(x^a_i)_{_{a\in A}}$ converges to a point $x_i\in X_i$. Let $\widetilde{x}=\{x_i\vert i\in I\}$. Then, it is easy to verify that $x^a\longrightarrow \widetilde{x}.$
\end{proof}

The following proposition is evident.
\begin{proposition}\label{612}{\rm A closed subspace of a $\mathcal{U}$-complete quasi-uniform space $(X,\mathcal{U})$ is $\mathcal{U}$-complete.}
\end{proposition}

\begin{proposition}\label{epi1}{\rm (see\cite[Theorem 1.7]{sto} Each quasi-uniform space $(X,\mathcal{U})$
can be embedded in a product of quasi-pseudometric spaces.
}
\end{proposition}

\begin{theorem}\label{epi}{\rm Any quasi-unifrom space $(X,\mathcal{U})$ has a $\mathcal{U}$-completion.
}
\end{theorem}
\begin{proof} By Proposition \ref{epi1}, we have that $(X,\mathcal{U})$ can be embedded in a product of
quasi-pseudometric spaces 
$\displaystyle\prod_{_i\in I}(X_{_i},d_{_i})$. Without loss of generality we may assume that
$d_i(x_i,y_i)\leq 1$ for all $x_i, y_i\in X_i$ and for all $i\in I$.
By Theorem \ref{a19} each space $(X_i,d_i)$ has a $(d_i)_{_S}$-completion $(\widehat{X}_{_i},\widehat{d}_{_i})$.
Therefore, $(X,\mathcal{U})$ can be embedded in 
$\displaystyle\prod_{_i\in I}(\widehat{X}_{_i},\widehat{d}_{_i})$. That is, there exists a quasi-uniformly
continuous mapping $\widehat{\phi}: (X,\mathcal{U})\longrightarrow \displaystyle\prod_{_i\in I}(\widehat{X}_{_i},\widehat{d}_{_i})$ such that $\widehat{\phi}(X)\subseteq \displaystyle\prod_{_i\in I}(\widehat{X}_{_i},\widehat{d}_{_i})$.
It follows from Propositions \ref{412},  \ref{213} and  \ref{612}
that $\overline{\widehat{\phi}(X)}\subseteq \displaystyle\prod_{_i\in I}(\widehat{X}_{_i},\widehat{d}_{_i})$, where $\overline{\widehat{\phi}(X)}$ is the closure of $\widehat{\phi}(X)$ in 
$\displaystyle\prod_{_i\in I}(\widehat{X}_{_i},\widehat{d}_{_i})$. Let 
\begin{equation}
\widehat{\mathcal{U}}=\{U\bigcap \overline{\widehat{\phi}(X)}\times \overline{\widehat{\phi}(X)}
\vert \ U\ {\rm is\ a\ member\ of\ the\ product\ quasi-uniformity\ for}\ \displaystyle\prod_{_i\in I}(\widehat{X}_{_i},\widehat{d}_{_i})\}. 
\label{eq:4515}
\end{equation}
Then, by Propositions \ref{412} and \ref{epi1}
we have that $(\widehat{X},\widehat{\mathcal{U}})$ is a $\widehat{\mathcal{U}}$-completion for
$(X,\mathcal{U})$.
\end{proof}

\section{Discussion} 
In the present paper, we give a new completion for quasi-metric and quasi-uniform spaces which generalizes the completion 
theories
of Doitchinov \cite{doi1} and Stoltenberg \cite{sto}. The main contribution in this paper is the use of the notion of the cuts of nets (sequences)
which generalizes the idea of Doitchinov 
for a completeness theory for quasi-metric and quasi-uniform spaces.
Doitchinov's
completeness theory for quiet spaces (a class of quasi-uniform spaces) is very well behaved and extends the completion theory of uniform spaces in a natural way.
According to Doitchinov \cite{doi1} a
natural definition of a Cauchy net
in a quasi-uniform space $(X,\mathcal{U})$ (similar conditions must be satisfied for a quasi-metric space $(X,d)$)
has to be defined in such a manner that the following
requirements are fulfilled:

(i) Every convergent sequence is a Cauchy net;

(ii) In the uniform case (i.e. when $(X,\mathcal{U})$
 the Cauchy nets are the usual ones.

Further a standard construction of a completion $(\widehat{X},\widehat{\mathcal{U}})$ of any quasi-uniform space $(X,\mathcal{U})$ 
should be possible such that:

(iii) If $(X,\mathcal{U})\subseteq (Y,\mathcal{V})$, 
where the inclusions are understood as quasi-uniform (resp. quasi-metric) embeddings and the second one is an extension of the former;

(iv) In the case when $(X,\mathcal{U})$ is a uniform space $(X^{\ast},\mathcal{U}^{\ast})$
is nothing but the usual uniform completion of 
$(X,\mathcal{U})$.

As it is well known, in metric spaces the notions of completeness by sequences and by nets agree and, further, the completeness of a metric space is equivalent to the completeness of the associated uniform space. 
Therefore, a natural requirement for a satisfactory theory of
completeness in quasi-metric spaces is the following:

(v)  In quasi-metric spaces the sequential completeness and completeness by nets agree.

The following definitions is due to Doitchinov ((see \cite[Condition Q and Definition 11]{doi1})) and 
inspired us to define the notion of cuts of nets.
More precisely: 

\begin{definition}\label{kj1}{\rm A quasi-uniform space $(X,\mathcal{U})$ is called {\it quiet} 
provided that for each 
$U\in\mathcal{U}$
there exists $V\in\mathcal{U}$ such that, if $x^{\prime}, x^{\prime\prime}\in X$
and $(x_a)_{a\in A}$ and $(x_{_\beta})_{_{\beta\in B}}$ are two nets in $X$, then from $(x,x_a)\in V$ for $a\in A$,
$(x_{_\beta},y)\in V$ for $\beta\in B$ and $(x_{_\beta},x_a)\rightarrow 0$ it follows that $(x,y)\in U$.
We say that $V$ is $Q$-{\it subordinated} to $U$.
}
\end{definition}
\begin{definition}\label{ak1}{\rm A net $(x_a)_{a\in A}$ in a quasi-uniform space is $D$-{\it Cauchy} if there exists
another net $(y_{_\beta})_{_{\beta\in B}}$ such that for each $U\in\mathcal{U}$, there exist $a_{_U}\in A$ 
$\beta_{_U}\in B$
satisfying 
$(y_{_\beta}, x_a)\in U$ whenever $a\geq a_{_U}\in A$ and
$\beta\geq\beta_{_U}\in B$. The space $(X,\mathcal{U})$ is $D$-{\it complete} if every $D$-Cauchy net in 
$X$ is convergent.
}
\end{definition}

\begin{definition}\label{kjz1}{\rm Two $D$-Cauchy nets $(x_a)_{a\in A}$ and $(x_{_\beta})_{_{\beta\in B}}$
are called {\it equivazent} if every conet of $(x_a)_{a\in A}$ is a conet of $(x_{_\beta})_{_{\beta\in B}}$ and vice versa.}
\end{definition}

\begin{proposition}\label{kjz}(See \cite[Proposition 12]{doi1}). {\rm If two $D$-Cauchy nets 
in a quiet quasi-uniform space $(X,\mathcal{U})$
have a common conet, then they are equivalent.}
\end{proposition}

Ιn the quasi-metric case the notions of $D$-completeness by nets is defined by sequences as follows:
 
 \begin{definition}\label{xso}{\rm A sequence $(x_n)_{n\in\mathbb{N}}$ in the quasi-metric space $(X,d)$ is called $D$-{\it Cauchy sequence} provided that for any natural number $k$ there exist a $(y_{_k})_{_{k\in\mathbb{N}}}$ and an $N_k$ such that $d(y_{_k},x_a)<\displaystyle{1\over k}$ whenever $m,n>N_k$.}
 \end{definition}

The concept of $D$-Cauchy net (sequence) proposed by Doitchinov enables him to realize the program outlined above under the assumption of quietness.

Moreover, $D$-completeness satisfies requirement (v). The following definition shows this fact.

\begin{proposition}\label{thr}{\rm (See \cite[Theorem 9]{doi1}).
In the balanced quasi-metric spaces the notions of $D$-completeness by sequences and by nets agree.}
\end{proposition}

\begin{proof}
If a quasi-metric space is $D$-complete, then each $D$-Cauchy net converges in $X$, and thus each $D$-Cauchy sequence converges in $X$. Conversely,
we prove that the sequential $D$-completeness implies that every $D$-Cauchy net in $X$ is convergent. Let $(X,d)$ be a balanced quasi-metric space in which any $D$-Cauchy sequence converges and let $(x_a)_{a\in A}$ be a $D$-Cauchy net in $(X,d)$. Then, $(x_a)_{a\in A}$ is a $D$-Cauchy net in the quasi-uniform space $\mathcal{U}_d$ generated by $d$.
By \cite[Page 208]{doi1}, $\mathcal{U}_d=\{U_{_\epsilon}\vert \epsilon>0\}$ where 
$U_{_\epsilon}=\{(x,y)\in X\times X\vert d(x,y)<\epsilon\}$
is a quiet quasi-uniform space. Since $(x_a)_{a\in A}$ is $D$-Cauchy there exists a net $(y_{_\beta})_{_{\beta\in B}}$ such that for any $U_{_\epsilon}\in \mathcal{U}_d$ there are $a_{_{U_{_\epsilon}}}\in A$ and $\beta_{_{U_{_\epsilon}}}\in B$ such that $(y_{_\beta},x_a)\in U$ 
whenever $a\geq a_{_{U_{_\epsilon}}}$, $\beta\geq \beta_{_{U_{_\epsilon}}}$ or equivalently $\displaystyle\lim_{a,\beta}d(y_{_\beta},x_a)=0$. 
Therefore, for each $n\in \mathbb{N}$, there exists $a_{n}\in A$, $\beta_{_n}\in B$ such that $(y_{_{\beta_{_n}}},x_{a_n})\in U_{_n}=\{(x,y)\in X\times X\vert d(x,y)<{1\over n}\}$.
It follows that $(x_{a_n})_{n\in \mathbb{N}}$ is a $D$-Cauchy sequence and thus it converges to a $x\in X$.
Therefore, $\displaystyle\lim_\beta(y_{_\beta},x)=0$ (\cite[Lemma 15]{doi1}).
Let $\epsilon>0$. Then, since $\mathcal{U}_d$ is quiet there exists $\epsilon^{\prime}>0$ such that $U_{\epsilon^{\prime}}$ is $Q$-subordinated to $U_\epsilon$. Let $n\in\mathbb{N}$ such that ${1\over n}<\epsilon^{\prime}$.
Then, from $(y_{_{\beta_{_n}}},x_a)\in U_{\epsilon^{\prime}}$, $(x,x)\in U_{\epsilon^{\prime}}$ and $\displaystyle\lim_\beta(y_{_\beta},x)=0$ we conclude that $(x,x_a)\in U_\epsilon$. It follows that $(x_a)_{a\in A}$ converges to $x$.
\end{proof}

By Proposition \ref{kjz}, to each $D$-Cauchy net $(x_a)_{a\in A}$, it corresponds a set of pairs of nets-conets which lead to the notion of cut of nets, that is, a pair $(\mathcal{C},\mathcal{D})$ where $\mathcal{C}$ contains all equivalent nets of $(x_a)_{a\in A}$ and $\mathcal{D}$ contains all of their conets.
The notion of $\mathcal{U}$-cut defined in this paper is a cut of nets $(\mathcal{A},\mathcal{B})$ where the members of $\mathcal{A}$ contains right $\mathcal{U}_{_S}$-Cauchy nets and $\mathcal{B}$ contains left $\mathcal{U}_{_S}$-Cauchy nets as they are defined by Stoltenberg.

At this point it is routine to check that all the requirements (i)-(iv) are satisfied for the proposed $\mathcal{U}$-completion.
The requirement (v) is also satisfied as shows Proposition \ref{pop}.

The validity of requirement (v) usually does not hold for some completion theories.
As remarked Stoltenberg \cite[Example 2.4]{sto}, there exists a sequentially right $K$-complete quasi-metric space $(X,d)$ which is not right $\mathcal{U}_{_S}$-complete. Actually, Stoltenberg 
\cite[Theorem 2.5]{sto} proved that the equivalence holds for a more general definition of a right $\mathcal{U}_{_S}$-Cauchy net. 
More precisely, Stoltenberg \cite{sto}  gives
a more general definition of a right $\mathcal{U}_{_S}$-Cauchy
net than that we use in this paper
as follows: 
A net $(x_{_a})_{_{a\in A}}$ in a quasi-unifrom space $(X,\mathcal{U})$ is called {\it right} (resp. {\it left}) $\mathcal{U}_{_{\mathcal{S}}}$-{\it Cauchy}
if for each $U\in\mathcal{U}$ there is $a_{_U}\in A$ such that $(x_{_\beta},x_{_\alpha})\in U$ (resp. $(x_{_\alpha},x_{_\beta})\in U$) whenever
$\alpha, \beta \in A$ and $a\geq \beta\geq  a_{_U}$.
However, using this definition one can find a quasi-metric space $(X,d)$ which is left $K$-sequentially complete but not $\mathcal{U}_{_S}$-complete.
He support this claim by offering the following example \cite[Example 2.4]{sto}:
Let $\mathcal{A}$ be the family of all countable subsets of the closed interval 
$[0,{1\over 3}]$ and let for every $A\in\mathcal{A}$ and $k\in\mathbb{N}$,
$X^{^A}_{_{k+1}}=A\cup \{{1\over 2},{3\over 4},...,{{2^{^k}-1}\over {2^{^k}}}\}$ and $X^{^A}_{_{\infty}}=\displaystyle\bigcup_{k\geq 1}X^{^A}_{_{k}}$.

Define 

\begin{center}
$\mathcal{X}_{_A}=\{X^{^A}_{_{k}}\vert k=1,2,...,\infty\}$ and 
$\mathfrak{J}=\{\bigcup\{\mathcal{X}_{_A}\vert A\in \mathcal{A}\}$.
\end{center}

Define $d:\mathfrak{J}\times \mathfrak{J}\rightarrow \mathbb{R}$ by
\begin{center} $d(X^{^A}_{_{k}},X^{^B}_{_{j}})={1\over {2^{^j}}}$ if $X^{^B}_{_{j}}\subset X^{^A}_{_{k}}$, $k=1,2,...,\infty$, $j=1,2,...,$ \\ .
\\ $d(X^{^A}_{_{k}},X^{^A}_{_{k}})=0$\ 
and\ $d(X^{^A}_{_{k}},X^{^B}_{_{j}})=0$ otherwise.
\end{center}
Stoltenberg proves that $(\mathfrak{J},d)$ is left $K$-sequentially complete and not $d_{_S}$-complete.
To overcome the weaknesses of 
this definition, in order to develop his theory of $\mathcal{U}_{_S}$-completeness, he uses Definition 
\ref{a0}. 
In order to prove his main result of the construction of $\mathcal{U}_{_S}$-completion 
he uses Theorem 2.5 which says that: 
A quasi-metric space $(X,d)$ is $d_{_S}$-complete if and only if
every left $K$-Cauchy sequence in $(X,d)$ converges with respect to $\tau_{_d}$ in $X$.
On the other hand,
the notion of $d_{_S}$-completeness
plays a central role in the constructed $\mathcal{U}_{_S}$-completion,
since this completion is a subset of
$\displaystyle\prod_{_i\in I}(\widehat{X}_{_i},\widehat{d}_{_i})$
($(\widehat{X}_{_i},\widehat{d}_{_i})$ is the $(d_i)_{_S}$-completion of the $T_{_0}$ quasi-pseudometric space
$(X_i,d_i)$). 
Gregori and Ferrer \cite{GF}, using Example 2.4 of \cite{sto} as a counter-example,
showed
that Stoltenberg's result of Theorem 2.5, based on his $\mathcal{U}_{_S}$-Cauchy net (\cite[Definition 2.1]{sto}), is not valid in general. 
In fact, the authors define a net $\Phi$ in $(\mathfrak{J},d)$ as follows:
Let $D=\mathbb{N}\cup \{a,b\}$, where $\mathbb{N}$
 is the set of natural with the usual order and 
$a,b\notin \mathbb{N}$, $a\neq b$ and $a\geq k$ and $b\geq k$ for $k\in\mathbb{N}$, 
$a\geq b$,
$b\geq a$, $a\geq a$ and $b\geq b$. Cleraly, $D$ is a directed set. Then,
\begin{center}
$\Phi(k)=X^{^A}_{_k}$ for $k\in\mathbb{N}$, $\Phi(a)=X^{^A}_{_\infty}$, $\Phi(a)=X^{^B}_{_\infty}$, where $A, B\in \mathcal{A}$ and $A\subset B$.
\end{center}
Then, $\Phi$ is a right $d_{_S}$-Cauchy net. Indeed, let $0<\varepsilon<1$ and $\lambda_{_0}\in\mathbb{N}$ be such that 
$\displaystyle{1\over {2^{\lambda_{_0}}}}<\varepsilon$. We have $k\leq a$, $k\leq b$, $k\geq k_{_0}$
$\lambda_{_0}\leq a\leq b$ and $\lambda_{_0}\leq b\leq a$. The condition $\lambda\ngeq k$ can hold for some $k, \lambda\in I$,
$k, \lambda\geq k_{_0}$ in the following cases:
\par
($\mathfrak{a}$) $k, \lambda\in\mathbb{N}$, $k, \lambda\geq k_{_0}$, $k>\lambda$.
\par
($\mathfrak{b}$) $k=a$, $\lambda\in \mathbb{N}$, $k, \lambda\geq k_{_0}$.
\par
($\mathfrak{c}$) $k=a$, $\lambda\in \mathbb{N}$, $k, \lambda\geq k_{_0}$.

In the case ($\mathfrak{a}$), we have that $X^{^A}_{_\lambda}\subset X^{^A}_k$ and thus
\begin{center}
$d(\Phi(k),\Phi(\lambda))=d(x^A_k,x^A_{_\lambda})=\displaystyle{1\over {2^{^\lambda}}}<\displaystyle{1\over {2^{^{k_{_0}}}}}<\varepsilon$.
\end{center}

In the case ($\mathfrak{b}$), we have that  we have that $X^{^A}_{_\lambda}\subset X^{^A}_\infty$ and thus
\begin{center}
$d(\Phi(a),\Phi(\lambda))=d(x^A_\infty, x^A_{_\lambda})=\displaystyle{1\over {2^{^\lambda}}}<\displaystyle{1\over {2^{^{k_{_0}}}}}<\varepsilon$.
\end{center}

The case ($\mathfrak{c}$) is similar to ($\mathfrak{c}$).

On the other hand, $\Phi$ is not convergent in $(\mathfrak{J},d)$. Indeed, let $X\in \mathfrak{J}\setminus X^{^B}_{\infty}$.
Then, for each $k\in I$, $b\geq k$ holds and thus 
\begin{center}
$d(X,\varphi(b))=d(X,X^{^B}_\infty)=1$.
\end{center}
It concludes that for each $\varepsilon>0$, $X^{^B}_\infty\notin B(X,\varepsilon)$ holds. Therefore, for any $\varepsilon>0$
no final segment of $\Phi$ is contained in $B(X,\varepsilon)$ which implies that $\Phi$ does not converge to $X$.
Similarly, if $X=X^{^B}_\infty$, then $a\geq k$ for each $k\in I$ and $d(X^{^B}_\infty,X^{^A}_\infty)=1$.

Ηence the problem that arises here is that
the net $\Phi$, although non-convergent,
is a $\mathcal{U}_{_S}$-Cauchy net.
At this point we will look at Doitchinov's view on this problem using another definition of cauchyness he used. More precisely:
\begin{definition}\label{ep}(See \cite[Definition, Page 129]{doi}).
{\rm Let $(X,d)$ be a quasi-pseudometric space.
A sequence
$(x_n)_{n\in\mathbb{N}}$ is called {\it Cauchy sequence} if for every
natural number $k$ there are a $y_{_k}\in X$ and an
$N_k\in\mathbb{N}$ such that $d(y_k,x_n)<\displaystyle{1\over k}$ when $n>N_{_k}$.
}
\end{definition}
According to Doitchinov a non-formal objection to this definition of Cauchyness is illustrated by the following example.
\begin{example}{\rm (Sorgenfrey line). Let $\mathbb{R}$ be the real line equipped with the quasi-metric
\par
\begin{center}

$d(x,y)=\left\{
\begin{array}{cc}
y-x & \mbox{if $x\leq y$,} \\
1   & \mbox{if $x>y$.}%
\end{array}%
\right. $

\end{center}
\smallskip
\par\noindent

}
\end{example}

For each $x \in X$, the collection $\{[x,x+r) \vert r> 0\}$ form a local base at the point $x$ for the topology generated by $d$ in $X$.

According to Doitchinov \cite[Page 130]{doi}:
{\it The sequence $(-{1\over n})_{n\in \mathbb{N}}$, although nonconvergent,
is a Cauchy sequence in the sense of Definition \ref{ep}. However, in view of the special character of the topology 
on the space (R,d), it seems very inconvenient to regard this sequence as a potentially convergent one, i.e. as one that could be made convergent by completing the space.}

As we describe above,
in order to avoid this unwanted phenomenon, Doitchinov has been introduced the $D$-completeness (see Definition \ref{ak1}).

On the other hand,
Gregori and Ferrer \cite{GF} also proposed a new definition of a right $\mathcal{U}_{_S}$-Cauchy net, 
for which the equivalence to sequential completeness holds (requirement (v)).

\begin{definition}\label{i1}{\rm A net $(x_a)_{a\in A}$ in a quasi-metric space $(X,d)$ is called $GF$-Cauchy if one of the following conditions holds:
\par
(i) For every maximal element $a^{\ast}\in A$ the net $(x_a)_{a\in A}$ converges to $x_{a^{\ast}}$; 
\par
(ii) $A$ has no maximal elements and $(x_a)_{a\in A}$ converges in $X$;
\par
(iii) $A$ has no maximal elements and $(x_a)_{a\in A}$ satisfies Definition 2.1 of \cite{sto}.
}
\end{definition}

More recently Cobzas \cite{cob}
has contributed
new results in Stoltenberg's completion theory. 
In order to avoid the shortcomings of the preorder relation, as, for instance, those put in evidence by 
Example of Gregori and Ferrer, he proposes a new definition of right $K$-Cauchy net in a quasi-metric 
space for which the corresponding 
completeness is equivalent to the sequential completeness.

\begin{definition}{\rm A net $(x_a)_{a\in A}$ in a quasi-metric space $(X, d)$ is called {\it strongly Stoltenberg-Cauchy} if for every
$\epsilon>0$
there exists $a_\epsilon\in A$ such that, for all $a, \beta\geq a_\epsilon$ ($\beta\leq a\vee a\nsim b$ implies that
$d(x_a,x_{_\beta})<\epsilon$ where $a\vee a\nsim b$ means that $a,\beta$ are incomparable (that is, 
no one of the relations $a\leq \beta$ or $\beta\leq a$ holds).
}
\end{definition}

It is interesting to see how $\mathcal{U}$-Competion will become if we change the members of the classes
of a $\mathcal{U}$-cuts with Gregori and Ferrer Cauchy nets or Cobzas Cauchy nets respectively.

\end{document}